\documentclass[11pt]{article}
\usepackage[a4paper,margin=28mm]{geometry}
\usepackage[T1]{fontenc}
\usepackage{lmodern}
\usepackage{amsmath,amssymb,amsthm,mathtools,mathrsfs,bm}
\DeclareFontFamily{U}{rsfs}{}
\DeclareFontShape{U}{rsfs}{m}{n}{<->rsfs10}{}
\usepackage{booktabs,array,longtable,tabularx,multirow,enumitem}
\usepackage{tikz,xcolor,graphicx,caption,subcaption,placeins,needspace}
\usepackage{float}
\usetikzlibrary{calc,positioning,arrows.meta,patterns,matrix}
\usepackage[authoryear,round]{natbib}
\usepackage[hidelinks]{hyperref}
\definecolor{navy}{HTML}{173657}
\definecolor{ink}{HTML}{192D40}
\definecolor{muted}{HTML}{657583}
\definecolor{rule}{HTML}{D9E0E5}
\definecolor{blue}{HTML}{326B9C}
\definecolor{bluewash}{HTML}{E5EFF8}
\definecolor{amber}{HTML}{A66E28}
\definecolor{amberwash}{HTML}{F7EEDC}
\definecolor{sand}{HTML}{DFA861}
\definecolor{brick}{HTML}{B95A45}
\definecolor{brickwash}{HTML}{F7E9E3}
\definecolor{teal}{HTML}{187D78}
\definecolor{tealwash}{HTML}{E1F0EB}
\definecolor{greywash}{HTML}{F2F4F5}
\definecolor{soft}{HTML}{F5F8FA}
\newcommand{\PP}{\mathcal P}

\newcommand{\GF}{\mathsf{GF}}
\newcommand{\GG}{\mathsf{G}}
\newcommand{\UU}{\mathsf{U}}
\newcommand{\JJ}{\mathsf{J}}

\newcommand{\R}{\mathbb R}
\renewcommand{\epsilon}{\varepsilon}

\newcommand{\HH}{\mathcal H}

\newcommand{\coker}{\operatorname{coker}}

\newcommand{\eps}{\varepsilon}
\newcommand{\doi}[1]{\href{https://doi.org/#1}{\nolinkurl{doi:#1}}}

\newcommand{\RR}{\mathbb R}

\newcommand{\calA}{\mathcal A}
\newcommand{\calD}{\mathscr D}

\newcommand{\calC}{\mathcal C}

\newcommand{\calT}{\mathcal T}
\newcommand{\calZ}{\mathcal Z}

\newcommand{\rank}{\operatorname{rank}}

\newcommand{\ve}{\Delta_0^{\circ}}
\newcommand{\ee}{\Delta_1^{\circ}}

\newtheorem{theorem}{Theorem}[section]
\newtheorem{lemma}[theorem]{Lemma}
\newtheorem{proposition}[theorem]{Proposition}
\newtheorem{corollary}[theorem]{Corollary}
\theoremstyle{definition}
\newtheorem{definition}[theorem]{Definition}
\newtheorem{example}[theorem]{Example}
\theoremstyle{remark}
\newtheorem{remark}[theorem]{Remark}
\newtheorem{question}[theorem]{Question}
\numberwithin{equation}{section}
\setlist{itemsep=3pt,topsep=4pt,leftmargin=2em}
\hypersetup{pdftitle={The Dimension Landscape of Bivariate C2 Splines},pdfauthor={Xinyu Wu and Jiansong Deng}}
\title{The Dimension Landscape of Bivariate $C^2$ Splines}
\author{Xinyu Wu\\{\small\texttt{xinyu97@ustc.edu.cn}}
\and Jiansong Deng\\{\small\texttt{dengjs@ustc.edu.cn}}}
\date{}
\begin{document}
\maketitle
\begin{abstract}
Schumaker's lower bound incorporates local geometric corrections and provides insight into the dimensions of bivariate spline spaces. Building on this bound, we study bivariate $C^2$ spline dimensions from three perspectives: universal attainment, generic attainment, and geometric jumps. We prove universal attainment in degree seven for every legal realization of every triangulation. We prove generic attainment in degrees five and six for every fixed combinatorial triangulation and give explicit polynomial conditions certifying generic failure on fixed triangulations in degrees three and four. For geometric jumps with matching local direction types at corresponding interior vertices, we prove nonexistence in degree seven and construct examples in degrees three and five. In degree six, universal attainment and the existence of such geometric jumps remain open questions.
\end{abstract}

\section{Introduction}\label{sec:intro}
Let $\mathcal T$ be a finite conforming triangulation of a connected polygonal disk $\Omega\subset\R^2$. For integers $d,r\ge0$, define
\[
S_d^r(\mathcal T)
:=\{s\in C^r(\Omega):s|_\tau\in\PP_d(\tau)\ \text{for every triangle }\tau\in\mathcal T\},
\]
where $\PP_d(\tau)$ denotes the restrictions to $\tau$ of bivariate polynomials of total degree at most $d$. For fixed $d$ and $r$, the dimension may depend on both the combinatorial structure of $\mathcal T$ and the positions of its vertices.

\citet{Strang73} initiated the study of this dimension problem and proposed a combinatorial dimension count for $C^1$ piecewise polynomials. \citet{Schumaker79} established a lower bound for bivariate splines on planar triangulations, with local correction terms determined by the number of distinct supporting directions at each interior vertex. We use this bound, denoted by $L_d^r(\mathcal T)$, as the common baseline throughout the paper. Thus $\dim S_d^r(\mathcal T)\ge L_d^r(\mathcal T)$, with the local corrections accounting for dependencies among the smoothness conditions at a vertex.

Research on spline dimensions has developed along three main lines: universal dimension formulas valid for every geometric realization, generic dimension formulas on each fixed combinatorial triangulation, and geometric jumps between realizations with the same combinatorics and vertexwise local direction types.

\paragraph{Universal dimension formulas.}
The universal question asks when $\dim S_d^r(\mathcal T)=L_d^r(\mathcal T)$ for every triangulation and every legal geometric realization. \citet{AlfeldSchumaker87} proved the formula for $d\ge4r+1$, and \citet{Hong91} extended the range to $d\ge3r+2$. For $C^1$ splines, \citet*{AlfeldPiperSchumaker87} established the quartic case; together with the higher-degree formulas, this gives universal attainment for every $d\ge4$. For $C^2$ splines, Hong's result gives universal attainment from degree eight onward.

\paragraph{Generic attainment and failure.}
The generic question fixes the combinatorial triangulation and varies all vertex coordinates. A property holds \emph{generically} if it holds at every legal realization outside the zero set of a suitable nonzero coordinate polynomial. \citet{Billera88} established the generic $C^1$ dimension formula for every $d\ge2$. For general smoothness, \citet{AlfeldSchumaker90} proved generic attainment at $d=3r+1$. Combined with the universal formulas in higher degrees, this gives generic attainment for all $d\ge3r+1$, hence for $C^2$ splines from degree seven onward.

\citet[Section~7]{DiPasqualeVillamizar21} discussed chord and boundary-pair constructions that produce splines beyond the global polynomial space even at generic vertex positions. For $C^2$ splines, these constructions yield cubic and quartic splines, respectively.

\begin{figure}[p]
\centering
\resizebox{\linewidth}{!}{
\begin{tikzpicture}[x=1cm,y=1cm,text=ink,
  every node/.style={inner sep=0pt,font=\fontsize{12}{14}\selectfont},
  phasecell/.style={line width=.9pt,rounded corners=3pt},
  jumpcell/.style={line width=.9pt,line join=round},
  formulacell/.style={line width=.9pt,line join=round},
  guide/.style={draw=black!20,dash pattern=on 3pt off 3pt,line width=.45pt},
  priorlink/.style={draw=knownresult!75,line width=.45pt,line cap=round},
  priorcite/.style={font=\fontsize{9.5}{11}\selectfont,text=ink!85,
    align=center,fill=white,inner sep=1.2pt}]
\definecolor{knownresult}{HTML}{B48A08}
\foreach \x in {4.75,8.25,11.75,15.25}{
  \draw[guide] (\x,1.98)--(\x,8.85);
}
\draw[brick,line width=1.1pt,dash pattern=on 6pt off 4pt]
  (4.75,1.95)--(4.75,9.54);
\node[text=brick,font=\fontsize{13}{15}\selectfont] at (4.75,9.99)
  {$d=2r$};
\draw[ink!75,line width=.85pt,-{Stealth[length=3mm]}] (4.08,8.85)--(16.42,8.85);
\foreach \x/\deg in {4.75/2,8.25/3,11.75/4,15.25/5}{
  \fill[ink!75] (\x,8.85) circle (.065);
  \node[fill=white,inner sep=2pt] at (\x,9.26) {$\deg$};
}
\node[anchor=west] at (16.57,8.85) {degree $d$};
\node[anchor=west,font=\fontsize{12}{15}\selectfont\bfseries] at (.42,7.58) {Universal formula};
\node[anchor=west,font=\fontsize{10}{12}\selectfont] at (.42,7.12) {All realizations};
\node[anchor=west,font=\fontsize{12}{15}\selectfont\bfseries] at (.42,5.36) {Generic phase};
\node[anchor=west,font=\fontsize{10}{12}\selectfont] at (.42,4.88) {Generic realizations};
\node[anchor=west,font=\fontsize{12}{15}\selectfont\bfseries] at (.42,2.80) {Geometric jump};
\node[anchor=west,font=\fontsize{10}{12}\selectfont] at (.42,2.34) {Fixed local data};
\foreach \x/\state in {4.75/\times,8.25/\times,11.75/\UU,15.25/\UU}{
  \draw[formulacell,draw=knownresult,fill=knownresult!10]
    (\x-.52,7.05)--(\x+.52,7.05)--(\x,7.96)--cycle;
  \node[text=knownresult,font=\fontsize{16}{18}\selectfont] at (\x,7.39) {$\state$};
}
\foreach \x in {4.75,8.25,11.75,15.25}{
  \draw[phasecell,draw=knownresult,fill=knownresult!10] (\x-.54,4.81) rectangle (\x+.54,5.55);
  \node[text=knownresult,font=\fontsize{16}{18}\selectfont] at (\x,5.18) {$\GG$};
}
\foreach \x in {4.75,8.25}{
  \draw[jumpcell,draw=knownresult,fill=knownresult!10]
    (\x-.54,2.62)--(\x,3.16)--(\x+.54,2.62)--(\x,2.08)--cycle;
  \node[text=knownresult,font=\fontsize{16}{18}\selectfont] at (\x,2.62) {$\JJ$};
}
\foreach \x in {11.75,15.25}{
  \draw[jumpcell,draw=knownresult,fill=knownresult!10]
    (\x-.54,2.62)--(\x,3.16)--(\x+.54,2.62)--(\x,2.08)--cycle;
  \node[text=knownresult,font=\fontsize{12}{14}\selectfont] at (\x,2.62) {no $\JJ$};
}
\foreach \x in {4.75,8.25}{
  \draw[priorlink] (\x,6.96)--(\x,6.80);
}
\node[priorcite,anchor=north] at (4.75,6.65) {Morgan--Scott\\\citeyearpar{MorganScott75}};
\node[priorcite,anchor=north] at (8.25,6.65) {Wu--Deng\\\citeyearpar{WuDeng26}};
\draw[priorlink] (4.75,4.70)--(4.75,4.54)--(8.25,4.54)--(8.25,4.70);
\node[priorcite,anchor=north] at (6.50,4.39) {Billera\\\citeyearpar{Billera88}};
\node[priorcite,anchor=north] at (4.75,1.75) {Morgan--Scott\\\citeyearpar{MorganScott75}};
\node[priorcite,anchor=north] at (8.25,1.75) {Wu--Deng\\\citeyearpar{WuDeng26}};
\draw[priorlink] (12.50,2.62)--(12.50,7.39);
\foreach \y in {7.39,5.18,2.62}{
  \draw[priorlink] (12.38,\y)--(12.50,\y);
}
\draw[priorlink] (12.50,5.18)--(12.65,5.18);
\node[priorcite,anchor=west] at (12.75,5.18)
  {Alfeld et al.\\\citeyearpar{AlfeldPiperSchumaker87}};
\draw[priorlink] (16.30,2.62)--(16.30,7.39);
\foreach \y in {7.39,5.18,2.62}{
  \draw[priorlink] (15.88,\y)--(16.30,\y);
}
\draw[priorlink] (16.30,5.18)--(16.63,5.18);
\node[priorcite,anchor=west] at (16.75,5.18) {Hong\\\citeyearpar{Hong91}};
\node[anchor=east,text=muted,font=\fontsize{10}{12}\selectfont] at (17.55,.95) {$\UU$ for all $d\ge5$};
\draw[draw=knownresult,fill=knownresult!10,line width=.8pt] (.42,.185) rectangle (.85,.515);
\node[anchor=west,font=\fontsize{10.5}{12.5}\selectfont] at (1.05,.35) {Prior work};
\end{tikzpicture}
}
\caption{The complete three-level phase classification for bivariate $C^1$ splines in degrees $d\ge2$. All entries are prior work, shown in yellow with representative author--year references. The quadratic entries originate in the manuscript of \citet{MorganScott75}; \citet{Shi91} characterizes the corresponding singularity. The red dashed line is $d=2r$; the three entries at $d=5$ remain valid for every higher degree.}
\label{fig:phase-c1}
\par\vspace{4mm}
\centering
\resizebox{\linewidth}{!}{
\begin{tikzpicture}[x=1cm,y=1cm,text=ink,
  every node/.style={inner sep=0pt,font=\fontsize{12}{14}\selectfont},
  phasecell/.style={line width=.9pt,rounded corners=3pt},
  jumpcell/.style={line width=.9pt,line join=round},
  formulacell/.style={line width=.9pt,line join=round},
  unresolved/.style={dash pattern=on 3pt off 2pt},
  guide/.style={draw=black!20,dash pattern=on 3pt off 3pt,line width=.45pt},
  priorlink/.style={draw=knownresult!75,line width=.45pt,line cap=round},
  workindex/.style={font=\fontsize{9.5}{11}\selectfont,text=newresult,fill=white,inner sep=1.2pt},
  priorcite/.style={font=\fontsize{9.5}{11}\selectfont,text=ink!85,
    align=center,fill=white,inner sep=1.2pt}]
\definecolor{newresult}{HTML}{326B9C}
\definecolor{knownresult}{HTML}{B48A08}
\definecolor{openresult}{HTML}{7D858C}
\foreach \x in {4.75,6.50,8.25,10.00,11.75,13.50}{
  \draw[guide] (\x,1.98)--(\x,8.85);
}
\draw[brick,line width=1.1pt,dash pattern=on 6pt off 4pt]
  (6.50,1.95)--(6.50,9.54);
\node[text=brick,font=\fontsize{13}{15}\selectfont] at (6.50,9.99)
  {$d=2r$};
\draw[ink!75,line width=.85pt,-{Stealth[length=3mm]}] (4.08,8.85)--(14.67,8.85);
\foreach \x/\deg in {4.75/3,6.50/4,8.25/5,10.00/6,11.75/7,13.50/8}{
  \fill[ink!75] (\x,8.85) circle (.065);
  \node[fill=white,inner sep=2pt] at (\x,9.26) {$\deg$};
}
\node[anchor=west] at (14.82,8.85) {degree $d$};
\node[anchor=west,font=\fontsize{12}{15}\selectfont\bfseries] at (.42,7.58) {Universal formula};
\node[anchor=west,font=\fontsize{10}{12}\selectfont] at (.42,7.12) {All realizations};
\node[anchor=west,font=\fontsize{12}{15}\selectfont\bfseries] at (.42,5.36) {Generic phase};
\node[anchor=west,font=\fontsize{10}{12}\selectfont] at (.42,4.88) {Generic realizations};
\node[anchor=west,font=\fontsize{12}{15}\selectfont\bfseries] at (.42,2.80) {Geometric jump};
\node[anchor=west,font=\fontsize{10}{12}\selectfont] at (.42,2.34) {Fixed local data};
\foreach \x/\state/\col in {4.75/\times/knownresult,6.50/\times/knownresult,8.25/\times/knownresult,11.75/\UU/newresult,13.50/\UU/knownresult}{
  \draw[formulacell,draw=\col,fill=\col!10]
    (\x-.52,7.05)--(\x+.52,7.05)--(\x,7.96)--cycle;
  \node[text=\col,font=\fontsize{16}{18}\selectfont] at (\x,7.39) {$\state$};
}
\draw[formulacell,unresolved,draw=openresult,fill=openresult!12]
  (9.48,7.05)--(10.52,7.05)--(10,7.96)--cycle;
\node[text=openresult,font=\fontsize{16}{18}\selectfont] at (10,7.39) {$?$};
\foreach \x/\state/\col in {4.75/\GF/knownresult,6.50/\GF/knownresult,8.25/\GG/newresult,10.00/\GG/newresult,11.75/\GG/knownresult,13.50/\GG/knownresult}{
  \draw[phasecell,draw=\col,fill=\col!10] (\x-.54,4.81) rectangle (\x+.54,5.55);
  \node[text=\col,font=\fontsize{16}{18}\selectfont] at (\x,5.18) {$\state$};
}
\foreach \x/\col in {4.75/newresult,6.50/knownresult,8.25/newresult}{
  \draw[jumpcell,draw=\col,fill=\col!10]
    (\x-.54,2.62)--(\x,3.16)--(\x+.54,2.62)--(\x,2.08)--cycle;
  \node[text=\col,font=\fontsize{16}{18}\selectfont] at (\x,2.62) {$\JJ$};
}
\draw[jumpcell,unresolved,draw=openresult,fill=openresult!12]
  (9.46,2.62)--(10,3.16)--(10.54,2.62)--(10,2.08)--cycle;
\node[text=openresult,font=\fontsize{16}{18}\selectfont] at (10,2.62) {$?$};
\foreach \x/\col in {11.75/newresult,13.50/knownresult}{
  \draw[jumpcell,draw=\col,fill=\col!10]
    (\x-.54,2.62)--(\x,3.16)--(\x+.54,2.62)--(\x,2.08)--cycle;
  \node[text=\col,font=\fontsize{12}{14}\selectfont] at (\x,2.62) {no $\JJ$};
}
\draw[draw=newresult!75,line width=.45pt] (11.75,6.96)--(11.75,6.80);
\node[workindex,anchor=north] at (11.75,6.65) {\S~\ref{sec:universal}};
\draw[draw=newresult!75,line width=.45pt]
  (8.25,4.70)--(8.25,4.54)--(10.00,4.54)--(10.00,4.70);
\node[workindex,anchor=north] at (9.125,4.39) {\S~\ref{sec:generic-attainment}};
\node[workindex,anchor=north] at (4.75,1.75) {\S~\ref{sec:jump-cubic}};
\node[workindex,anchor=north] at (8.25,1.75) {\S~\ref{sec:jump-quintic}};
\node[workindex,anchor=north] at (11.75,1.75) {\S~\ref{sec:universal}};
\draw[priorlink] (4.75,6.96)--(4.75,6.80)--(6.50,6.80)--(6.50,6.96);
\node[priorcite,anchor=north] at (5.35,6.65)
  {DiPasquale--Villamizar\\\citeyearpar{DiPasqualeVillamizar21}};
\draw[priorlink] (8.25,6.96)--(8.25,6.80);
\node[priorcite,anchor=north] at (8.40,6.65)
  {Yuan--Stillman\\\citeyearpar{YuanStillman19}};
\draw[priorlink] (4.75,4.70)--(4.75,4.54)--(6.50,4.54)--(6.50,4.70);
\node[priorcite,anchor=north] at (5.625,4.39)
  {DiPasquale--Villamizar\\\citeyearpar{DiPasqualeVillamizar21}};
\draw[priorlink] (11.75,4.70)--(11.75,4.54);
\node[priorcite,anchor=north] at (11.75,4.39)
  {Alfeld--Schumaker\\\citeyearpar{AlfeldSchumaker90}};
\node[priorcite,anchor=north] at (6.50,1.75)
  {Diener\\\citeyearpar{Diener90}};
\draw[priorlink] (14.55,2.62)--(14.55,7.39);
\foreach \y in {7.39,5.18,2.62}{
  \draw[priorlink] (14.13,\y)--(14.55,\y);
}
\draw[priorlink] (14.55,5.18)--(14.88,5.18);
\node[priorcite,anchor=west] at (15.00,5.18)
  {Hong\\\citeyearpar{Hong91}};
\node[anchor=east,text=muted,font=\fontsize{10}{12}\selectfont] at (15.80,.95)
  {$\UU$ for all $d\ge8$};
\draw[draw=newresult,fill=newresult!10,line width=.8pt] (.42,.185) rectangle (.85,.515);
\node[anchor=west,font=\fontsize{10.5}{12.5}\selectfont] at (1.05,.35) {This work};
\draw[draw=knownresult,fill=knownresult!10,line width=.8pt] (4.15,.185) rectangle (4.58,.515);
\node[anchor=west,font=\fontsize{10.5}{12.5}\selectfont] at (4.78,.35) {Prior work};
\draw[draw=openresult,fill=openresult!12,line width=.8pt] (8.30,.185) rectangle (8.73,.515);
\node[anchor=west,font=\fontsize{10.5}{12.5}\selectfont] at (8.93,.35) {Open};
\end{tikzpicture}
}
\caption{The three-level dimension landscape for bivariate $C^2$ splines. Blue indicates results of this work, with section references below the symbols; yellow indicates prior work with representative author--year references; gray indicates open questions. The quartic jump uses the Morgan--Scott configuration \citep{MorganScott75}, with its degree-$2r$ instability due to \citet{Diener90}. The red dashed line is $d=2r$. Explicit fixed-combinatorics witnesses for the low-degree $\GF$ entries are given in Section~\ref{sec:generic-failure}.}
\label{fig:phase}
\end{figure}

\paragraph{Geometric jumps with fixed local direction types.}
The jump question compares two realizations of the same combinatorial triangulation whose local direction types agree at corresponding interior vertices. The manuscript of \citet{MorganScott75} contains the classical example in which the dimension of $S_2^1$ changes from six to seven while these local data remain fixed. \citet{Shi91} characterized its geometric singularity. \citet{Diener90} studied dimension instability in degree $2r$, and \citet*{DengFengKozak00} analyzed the geometry of Morgan--Scott spaces in this degree.

Recent preprints extend this line of work. \citet{WuDeng26} constructed a $C^1$ cubic realization pair with different dimensions and the same supporting-line counts at corresponding interior vertices. \citet{Potu26} subsequently gave another counterexample for $S_3^1$, which can also be used to obtain a jump with fixed local data. \citet*{WuZhengDeng26} gave a geometric criterion for Morgan--Scott spaces $S_{2r}^r$, identifying concurrence of auxiliary lines as a source of an extra dimension. These results exhibit dimension changes arising from global geometric compatibility while the local direction data remain fixed.

\paragraph{Dimension excess and homological methods.}
Strict excess on a specified geometric triangulation has also been studied directly. \citet{Tohaneanu05} proved excess at $d=2r$ for every $r\ge1$ on a particular geometric triangulation with prescribed collinearities. \citet{YuanStillman19} produced a $C^2$ quintic counterexample to the Schenck--Stiller $2r+1$ conjecture. \citet*{SchenckStillmanYuan20} showed that no constant $c$, independent of $r$, makes $d\ge2r+c$ sufficient for universal attainment over all triangulations and smoothness orders.

A single realization with strict excess disproves universal attainment at the given degree, but does not by itself establish generic failure or a geometric jump. Generic failure requires strict excess for generic realizations of a fixed combinatorial triangulation. A geometric jump requires two realizations of the same triangulation with different dimensions and the same local direction types at corresponding interior vertices.

Homological methods provide an algebraic framework for dimension calculations. \citet{SchenckStillman97} used local cohomology to study the dimension defect across polynomial degrees. \citet{MourrainVillamizar13} developed dimension bounds and a short proof of the formula for $d\ge4r+1$. \citet{ToshniwalDiPasquale21} recovered the range $d\ge3r+2$ in a framework allowing mixed smoothness and more general mesh topologies. These works relate dimension defects to local ideals and global compatibility, providing the background for the local-ideal and cofactor formulations used below.

For $C^1$ splines in degrees $d\ge2$, the phase classification at these three levels is complete. Universal attainment holds exactly for $d\ge4$ \citep{AlfeldPiperSchumaker87,Hong91}, generic attainment holds for every $d\ge2$ \citep{Billera88}, and fixed-local-data jumps occur exactly in degrees two and three \citep{MorganScott75,Shi91,WuDeng26}. Figure~\ref{fig:phase-c1} summarizes this classification.

We therefore focus on $C^2$ splines. The classical formulas give universal attainment for $d\ge8$ and generic attainment for $d\ge7$, while the low-degree constructions above establish generic failure and geometric jumps in parts of the remaining range. Figure~\ref{fig:phase} places these results alongside the new results of this paper and the unresolved cases.

\Needspace{22\baselineskip}
Our contributions to the $C^2$ dimension landscape are as follows:
\begin{enumerate}[label=\textup{(\roman*)}]
\item \textbf{Universal attainment in degree seven.} Theorem~\ref{thm:universal} proves attainment for every legal realization, including coincident local supporting directions. Together with Hong's theorem for $d\ge8$, it gives universal attainment for every $d\ge7$. Fixed-local-data jumps are consequently excluded in degree seven.
\item \textbf{Generic attainment in degrees five and six.} Theorems~\ref{thm:generic} and~\ref{thm:gthreshold} prove attainment for every fixed combinatorial triangulation in these degrees. With the classical higher-degree results and the low-degree generic failures, this determines the exact generic threshold as five.
\item \textbf{Geometric jumps in degrees three and five.} Theorem~\ref{thm:jumps} gives explicit rational realization pairs with different dimensions and the same local direction types at corresponding interior vertices. The quintic pair also preserves all orientation signs.
\item \textbf{Two open problems in degree six.} We formulate the problems of universal attainment of the lower bound and the existence of geometric jumps with fixed local data for $S_6^2$ \textup{(Questions~\ref{q:s6} and~\ref{q:j6})}. We explain their logical relation and identify six and seven as the only possible universal thresholds.
\end{enumerate}
For the known cubic and quartic generic-failure mechanisms and the classical quartic jump, we provide explicit coordinates and short proofs. Theorem~\ref{thm:generic-failure-criterion} also characterizes generic failure by a positive generic defect on one fixed triangulation.

Section~\ref{sec:prelim} introduces the notation, lower bound, phase definitions, and generic-failure criterion. Sections~\ref{sec:universal}--\ref{sec:jumps} treat universal attainment, generic behavior, and geometric jumps, respectively. Section~\ref{sec:open} formulates the sextic questions, and Section~\ref{sec:conclusion} concludes. Appendices~\ref{app:u7} and~\ref{app:g5} give the complete proofs of universal septic attainment and generic quintic attainment.
\FloatBarrier

\section{Preliminaries}\label{sec:prelim}
This section fixes the notation and baseline for the dimension landscape. We first describe geometric realizations, spline spaces, and Schumaker's lower bound. We then define the universal, generic, and geometric-jump phases and characterize generic failure by a positive generic defect.

\subsection{Spline spaces and Schumaker's lower bound}\label{sec:splines-lower-bound}
Let $\Delta$ be a finite abstract triangulation of a disk, with vertex, edge, and triangle sets $\Delta_0$, $\Delta_1$, and $\Delta_2$, respectively. Write $V=\Delta_0$, and let $V^\circ=\Delta_0^\circ$ and $E^\circ=\Delta_1^\circ$ denote the interior vertex and interior edge sets, respectively.

A \emph{legal realization} $p=(p_v)_{v\in V}$, with $p_v=(x_v,y_v)\in\R^2$, is a straight-line embedding: vertices are distinct, triangles are nondegenerate, and two simplices intersect only in their common face, if any. Let $\mathcal R_\Delta\subset\R^{2|V|}$ be the space of these realizations, and let $\mathcal T_p$ be the resulting conforming triangulation of the polygonal disk $\Omega_p$. In geometric expressions for a fixed realization, we identify a vertex $v$ with its point $p_v$.

All vertices may move; no boundary-value or normal-derivative conditions are imposed. For an integer $d\ge0$, define
\[
S_d^2(\Delta,p):=S_d^2(\mathcal T_p)
=\{s\in C^2(\Omega_p):s|_\tau\in\PP_d(\tau)\ \text{for every triangle }\tau\in\mathcal T_p\},
\]
where $\PP_d(\tau)$ denotes the restrictions to $\tau$ of real bivariate polynomials of total degree at most $d$. We write $\PP_d$ for the full polynomial space and $P_d=\dim\PP_d=\binom{d+2}{2}$. Write $I=|V^\circ|$, $E=|E^\circ|$, and $b=|V\setminus V^\circ|$ for the numbers of interior vertices, interior edges, and boundary vertices. An interior edge is not contained in the boundary; both endpoints may be boundary vertices. The valence $\deg_\Delta(v)$ is the number of edges incident to $v$. Let $g_v$ count their distinct supporting lines. Opposite rays on one line count as one direction.

For $d\ge3$, set
\begin{equation}\label{eq:sigma}
\sigma_d(g)=\sum_{j=1}^{d-2}\max\{3-(g-1)j,0\}.
\end{equation}
The lower bound of \citet{Schumaker79}, with the actual local direction corrections, is
\begin{equation}\label{eq:schumaker}
L_d^2=P_d+E\binom{d-1}{2}-I(P_d-6)+\sum_{v\in V^\circ}\sigma_d(g_v).
\end{equation}
We use this bound throughout and define the dimension defect by
\begin{equation}\label{eq:delta}
\delta_d^L(\Delta,p)=\dim S_d^2(\Delta,p)-L_d^2(\Delta,p).
\end{equation}

Let $N_j$ count interior vertices with exactly $j$ supporting directions. For $C^2$ smoothness, \eqref{eq:sigma} becomes
\begin{equation}\label{eq:localsimple}
\sigma_3(g)=\begin{cases}2&g=2,\\1&g=3,\\0&g\ge4,\end{cases}\qquad
\sigma_d(g)=\begin{cases}3&g=2,\\1&g=3,\\0&g\ge4,\end{cases}\quad d\ge4.
\end{equation}
Euler's identity for a disk gives $E=3I+b-3$ \citep[Section 2]{MourrainVillamizar13}. Substitution yields
\begin{equation}\label{eq:master}
\begin{split}
L_d^2-P_d&=(d-2)(d-4)I+\binom{d-1}{2}(b-3)+N_3+\kappa_dN_2,\\
\kappa_3&=2,\qquad\kappa_d=3\quad(d\ge4).
\end{split}
\end{equation}
Thus $L_d^2\ge P_d$ for $d\ge4$, whereas $L_3^2<P_3$ is possible. Table~\ref{tab:lower} lists the low-degree bounds.

\begin{table}[!htbp]\centering\small
\begin{tabular}{ccl}\toprule
$d$&$P_d$&$L_d^2$\\\midrule
3&10&$10-I+(b-3)+N_3+2N_2$\\
4&15&$15+3(b-3)+N_3+3N_2$\\
5&21&$21+3I+6(b-3)+N_3+3N_2$\\
6&28&$28+8I+10(b-3)+N_3+3N_2$\\
7&36&$36+15I+15(b-3)+N_3+3N_2$\\
8&45&$45+24I+21(b-3)+N_3+3N_2$\\\bottomrule
\end{tabular}
\caption{Lower bounds using the actual local geometry, including coincident supporting directions.}\label{tab:lower}
\end{table}
\FloatBarrier

Since every global cubic polynomial is a $C^2$ spline, we have $\dim S_3^2\ge P_3=10$. Table~\ref{tab:lower} therefore implies
\[
\dim S_3^2\ge 10>L_3^2
\qquad\text{whenever}\qquad
I>(b-3)+N_3+2N_2.
\]
Thus, in degree three, strict excess over Schumaker's bound may already be forced by the dimension of the global polynomial space. Example~\ref{ex:gf3} illustrates this situation.

\subsection{Phase definitions}\label{sec:phase-definitions}
Following \citet{Billera88}, we say that a property $\mathcal P$ \emph{holds generically} for realizations of a fixed abstract triangulation $\Delta$ if there is a nonzero real polynomial $H_\Delta$ in the coordinates of all vertices such that
\[
H_\Delta(p)\ne0\quad\Longrightarrow\quad\mathcal P(\Delta,p)
\qquad\text{for every }p\in\mathcal R_\Delta.
\]
Using this convention, we define the four phases below.

\begin{definition}[Dimension phases]\label{def:phase}
Fix $(d,2)$ with $d\ge3$.
\begin{description}[leftmargin=13mm,labelwidth=10mm,style=nextline]
\item[$\UU$] \textbf{Universal attainment:} $\delta_d^L=0$ on every realization of every triangulation.
\item[$\GG$] \textbf{Generic attainment:} for every fixed abstract triangulation $\Delta$, some nonzero coordinate polynomial $H_{\Delta,d}$ makes every $p\in\mathcal R_\Delta$ with $H_{\Delta,d}(p)\ne0$ satisfy $\delta_d^L(\Delta,p)=0$.
\item[$\GF$] \textbf{Generic failure:} the negation of $\GG$; that is, $(d,2)\in\GF$ if and only if $(d,2)\notin\GG$.
\item[$\JJ$] \textbf{Geometric jump with fixed local data:} some fixed abstract triangulation has two realizations with equal local direction types at corresponding interior vertices but different $\dim S_d^2$.
\end{description}
Here the \emph{local direction type} of an interior vertex is one of the three cases $g_v=2$, $g_v=3$, or $g_v\ge4$.
\end{definition}

Accordingly, $(d,r)\in\UU$, $(d,r)\in\GG$, $(d,r)\in\GF$, and $(d,r)\in\JJ$ mean, respectively, universal attainment, generic attainment, generic failure, and the existence of a fixed-local-data geometric jump at degree $d$ and smoothness $r$. From this point onward, we restrict these sets to $r=2$.

\begin{remark}[Fixed local data in a jump]\label{rem:fixed-local-data}
A pair witnessing $\JJ$ must preserve the local direction type at every corresponding interior vertex. In particular, $N_2$ and $N_3$ remain unchanged, so the two realizations have the same actual lower bound. A dimension difference caused by changing these counts is excluded from $\JJ$; Example~\ref{ex:u7local} illustrates this distinction. Equal totals $N_2,N_3$ say only that the numbers of two-direction and three-direction vertices agree. They do not specify which vertices have each type. Definition~\ref{def:phase} requires each corresponding vertex to retain its own type: two, three, or at least four supporting directions.
\end{remark}

\begin{remark}[Relations between the three levels]\label{rem:phase-relations}
Universal attainment, generic behavior, and geometric jumps are different properties with different quantifiers. They satisfy
\[
\UU\subseteq\GG,\qquad \UU\cap\JJ=\varnothing.
\]
Generic attainment and geometric jumps can coexist: a generic dimension formula may hold while special realizations of the same combinatorial triangulation have larger dimension. The quintic results in Sections~\ref{sec:generic} and~\ref{sec:jumps} give this coexistence.
\end{remark}

The next theorem characterizes generic failure by a positive generic defect on one fixed triangulation.

\begin{theorem}[Generic failure and positive generic defect]\label{thm:generic-failure-criterion}
Fix $d\ge3$. For each finite abstract triangulation $\Delta$ of a disk, there is a nonnegative integer $\delta_{\Delta,d}^{\mathrm{gen}}$ and a nonzero polynomial $Q_{\Delta,d}$ in all vertex coordinates such that
\begin{equation}\label{eq:generic-defect}
Q_{\Delta,d}(p)\ne0\quad\Longrightarrow\quad
\delta_d^L(\Delta,p)=\delta_{\Delta,d}^{\mathrm{gen}}
\qquad(p\in\mathcal R_\Delta).
\end{equation}
We call this integer the \emph{generic defect} of $\Delta$ in degree $d$. The following are equivalent:
\begin{enumerate}[label=\textup{(\roman*)}]
\item $(d,2)\in\GF$.
\item Some fixed triangulation $\Delta$ satisfies $\delta_{\Delta,d}^{\mathrm{gen}}>0$.
\item There exist a fixed triangulation $\Delta$ and a nonzero coordinate polynomial $H$ such that every legal realization with $H(p)\ne0$ satisfies $\delta_d^L(\Delta,p)>0$.
\end{enumerate}
\end{theorem}

\begin{proof}
Fix $\Delta$. We first establish existence by constructing the generic defect and its nonzero polynomial. For each interior edge $e=(a,b)$, take $\ell_e(x)=\det(p_b-p_a,x-p_a)$ and one smoothing cofactor $q_e\in\PP_{d-3}$, so that adjacent polynomial pieces satisfy $f_{\tau^+}-f_{\tau^-}=\ell_e^3q_e$. By the smoothing cofactor-conformality method \citep[Lemma~2.2 and Theorem~2.1]{ChuiWang83}, on a disk the compatible cofactor tuples correspond to splines modulo global polynomials. Their interior-vertex compatibility conditions are
\[
\sum_{e\ni v}\epsilon_{ve}\ell_e^3q_e\equiv0,
\qquad v\in V^\circ,
\]
where $\epsilon_{ve}\in\{1,-1\}$ are fixed signs from coherent oriented circuits around the interior vertices. The cofactor coefficients give $M=E\binom{d-1}{2}$ unknowns. Expanding the compatibility conditions in fixed monomial bases gives a matrix $C_d(p)$ with polynomial entries in the vertex coordinates. Hence, for every legal realization,
\begin{equation}\label{eq:polynomial-constraint-rank}
\dim S_d^2(\Delta,p)=P_d+M-\rank C_d(p).
\end{equation}

The ranks $\rank C_d(p)$ attained as $p$ ranges over $\mathcal R_\Delta$ form a nonempty finite set of integers. Let $\rho$ be their maximum and choose a legal $p_*$ attaining it. Select a $\rho$-minor $F$ of $C_d(p)$ nonzero at $p_*$, fixing its row and column indices; when $\rho=0$, take $F=1$. Then $F$ is a nonzero coordinate polynomial, and every legal $p$ with $F(p)\ne0$ satisfies $\rank C_d(p)=\rho$.

Choose any ordering of the vertices and set
\[
D_\Delta(p)=\prod_{v\in V^\circ}
\ \prod_{\substack{i<j\\ i,j\text{ adjacent to }v}}
\det(p_i-p_v,p_j-p_v),
\]
with an empty product equal to $1$. Each factor is a nonzero polynomial in the full coordinate space. Thus $D_\Delta\not\equiv0$, and $D_\Delta(p)\ne0$ forces all incident supporting lines at each interior vertex to be distinct. Consequently, $g_v=\deg_\Delta(v)$ and the actual lower bound has the fixed value
\[
L_{\Delta,d}^{\mathrm{gen}}
=P_d+E\binom{d-1}{2}-I(P_d-6)
+\sum_{v\in V^\circ}\sigma_d(\deg_\Delta(v)).
\]

Since $F(p_*)\ne0$, it remains nonzero under sufficiently small perturbations of $p_*$. The nonzero polynomial $D_\Delta$ cannot vanish on an entire neighborhood of $p_*$, so one can choose a sufficiently small legal perturbation $p$ with $D_\Delta(p)F(p)\ne0$. Set $Q_{\Delta,d}=D_\Delta F$. At every legal realization with $Q_{\Delta,d}(p)\ne0$,
\[
\delta_d^L(\Delta,p)
=P_d+M-\rho-L_{\Delta,d}^{\mathrm{gen}}
=:\delta_{\Delta,d}^{\mathrm{gen}}\ge0,
\]
where the inequality is Schumaker's lower bound. The constructed witness polynomial $Q_{\Delta,d}=D_\Delta F$ ensures that the defect is constant at legal realizations in its nonvanishing set. Thus the construction gives a pair $(\delta_{\Delta,d}^{\mathrm{gen}},Q_{\Delta,d})$ satisfying \eqref{eq:generic-defect}, establishing existence.

Next we prove uniqueness of the integer. Suppose a second pair $(\widetilde\delta,\widetilde Q)$ consists of a nonnegative integer and a nonzero coordinate polynomial such that every legal $p$ with $\widetilde Q(p)\ne0$ satisfies $\delta_d^L(\Delta,p)=\widetilde\delta$. The nonzero polynomial $Q_{\Delta,d}\widetilde Q$ cannot vanish on a nonempty open set of legal realizations, so some legal $p_0$ satisfies $Q_{\Delta,d}(p_0)\widetilde Q(p_0)\ne0$. There $\delta_{\Delta,d}^{\mathrm{gen}}=\delta_d^L(\Delta,p_0)=\widetilde\delta$, proving uniqueness of the generic defect.

For a fixed $\Delta$, generic attainment is equivalent to $\delta_{\Delta,d}^{\mathrm{gen}}=0$. Indeed, if this triangulation attains the bound generically, Definition~\ref{def:phase} provides a nonzero coordinate polynomial $K$ such that every legal $p$ with $K(p)\ne0$ satisfies $\delta_d^L(\Delta,p)=0$. The nonzero polynomial $KQ_{\Delta,d}$ is nonzero at some legal realization; there the defect is both zero and $\delta_{\Delta,d}^{\mathrm{gen}}$. Conversely, when this integer is zero, $Q_{\Delta,d}$ itself witnesses generic attainment. Therefore
\[
(d,2)\in\GG
\quad\Longleftrightarrow\quad
\delta_{\Delta,d}^{\mathrm{gen}}=0\ \text{for every }\Delta.
\]
Taking the negation and using nonnegativity proves \textup{(i)}$\Longleftrightarrow$\textup{(ii)}. If \textup{(ii)} holds, take $H=Q_{\Delta,d}$ to obtain \textup{(iii)}. Conversely, under \textup{(iii)}, the product $HQ_{\Delta,d}$ is nonzero at some legal realization, where \eqref{eq:generic-defect} gives $\delta_{\Delta,d}^{\mathrm{gen}}=\delta_d^L(\Delta,p)>0$. This proves \textup{(ii)}.
\end{proof}

\begin{remark}[Verifying generic failure]\label{rem:verifying-generic-failure}
The positive-generic-defect criterion is easier to verify than the literal negation of generic attainment: it suffices to find one fixed triangulation and one nonzero polynomial $H$ such that every legal realization with $H\ne0$ satisfies $\dim S_d^2>L_d^2$.
\end{remark}

\FloatBarrier

\section{Universal attainment}\label{sec:universal}
We establish universal attainment in degree seven while allowing coincident local supporting directions. A local-direction example then explains why different realizations can still have different dimensions when their actual lower bounds differ.

\begin{theorem}[Universal septic attainment]\label{thm:universal}
For every real straight-line realization of every finite triangulated disk, including coincident local supporting directions,
\begin{equation}\label{eq:u7main}
\dim S_7^2(\Delta,p)=L_7^2(\Delta,p)
=36+15E-30I+3N_2+N_3.
\end{equation}
\end{theorem}
Appendix~\ref{app:u7} proves the theorem by excluding a nonzero row relation on the local-ideal target. The proof passes from a dual state to its scalar support and rules out the resulting convex-hull star using cyclic cubic identities; Section~\ref{u7:sec:intro} gives the detailed outline.

Together with Hong's theorem for $d\ge8$ \citep{Hong91}, Theorem~\ref{thm:universal} gives universal attainment for every $d\ge7$, retaining the local corrections for coincident supporting directions. In particular, $(7,2)\notin\JJ$: two realizations with the same vertexwise local direction types have the same actual lower bound and hence, by Theorem~\ref{thm:universal}, the same dimension.

Universal attainment does not mean that every realization of one abstract triangulation has the same dimension: each realization attains its own actual lower bound, which may vary with $N_2,N_3$. The following example explains why the geometric-jump comparisons in Section~\ref{sec:jumps} require fixed local direction data, and hence unchanged $N_2,N_3$.
\begin{example}[A septic dimension change from local directions]\label{ex:u7local}
Take the square boundary $(-1,-1)$, $(1,-1)$, $(1,1)$, $(-1,1)$ and join its four corners to one interior vertex. The two realizations differ only in this interior coordinate:
\[
\begin{array}{c|c|c|c}
\text{Interior coordinate}&g_v&(I,E,b)&\dim S_7^2=L_7^2\\\hline
(0,0)&2&(1,4,4)&69\\
(1/4,1/8)&4&(1,4,4)&66
\end{array}
\]
See Figures~\ref{fig:u7local}\subref{fig:u7local-1} and~\ref{fig:u7local}\subref{fig:u7local-2}.
\end{example}
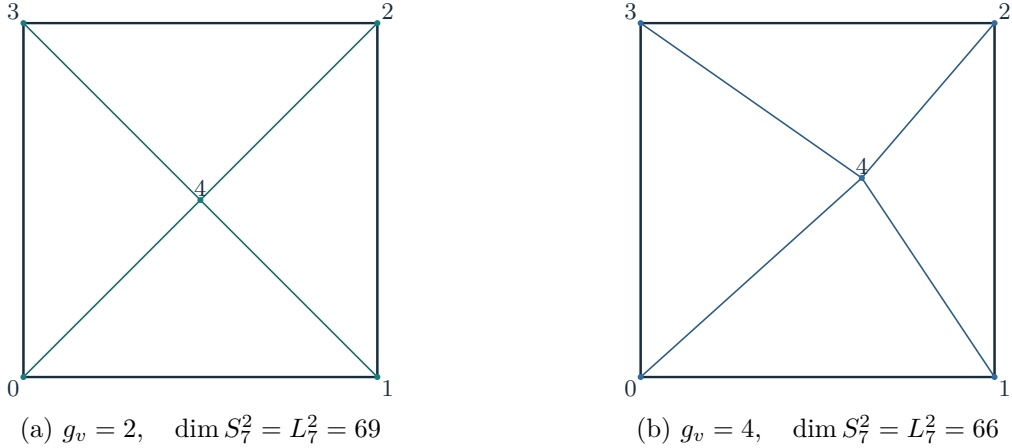
\begin{figure}[!htbp]\centering
\begin{subfigure}[t]{.47\linewidth}\centering
\begin{tikzpicture}[scale=.78,line cap=round,every node/.style={font=\footnotesize}]
\coordinate (v0) at (0.000000000,0.000000000);
\coordinate (v1) at (6.000000000,0.000000000);
\coordinate (v2) at (6.000000000,6.000000000);
\coordinate (v3) at (0.000000000,6.000000000);
\coordinate (v4) at (3.000000000,3.000000000);
\draw[teal!75!ink,line width=.6pt] (v0)--(v1);
\draw[teal!75!ink,line width=.6pt] (v0)--(v3);
\draw[teal!75!ink,line width=.6pt] (v0)--(v4);
\draw[teal!75!ink,line width=.6pt] (v1)--(v2);
\draw[teal!75!ink,line width=.6pt] (v1)--(v4);
\draw[teal!75!ink,line width=.6pt] (v2)--(v3);
\draw[teal!75!ink,line width=.6pt] (v2)--(v4);
\draw[teal!75!ink,line width=.6pt] (v3)--(v4);
\draw[ink,line width=.9pt] (v0)--(v1)--(v2)--(v3)--cycle;
\fill[teal] (v0) circle(1.45pt);
\node[below left,inner sep=1.3pt,text=ink] at (v0) {$ 0 $};
\fill[teal] (v1) circle(1.45pt);
\node[below right,inner sep=1.3pt,text=ink] at (v1) {$ 1 $};
\fill[teal] (v2) circle(1.45pt);
\node[above right,inner sep=1.3pt,text=ink] at (v2) {$ 2 $};
\fill[teal] (v3) circle(1.45pt);
\node[above left,inner sep=1.3pt,text=ink] at (v3) {$ 3 $};
\fill[teal] (v4) circle(1.45pt);
\node[above,inner sep=1.3pt,text=ink] at (v4) {$ 4 $};
\path[use as bounding box] (-.45,-.45) rectangle (6.45,6.45000000);
\end{tikzpicture}
\caption{$g_v=2,\quad\dim S_7^2=L_7^2=69$}\label{fig:u7local-1}
\end{subfigure}\hfill
\begin{subfigure}[t]{.47\linewidth}\centering
\begin{tikzpicture}[scale=.78,line cap=round,every node/.style={font=\footnotesize}]
\coordinate (v0) at (0.000000000,0.000000000);
\coordinate (v1) at (6.000000000,0.000000000);
\coordinate (v2) at (6.000000000,6.000000000);
\coordinate (v3) at (0.000000000,6.000000000);
\coordinate (v4) at (3.750000000,3.375000000);
\draw[blue!75!ink,line width=.6pt] (v0)--(v1);
\draw[blue!75!ink,line width=.6pt] (v0)--(v3);
\draw[blue!75!ink,line width=.6pt] (v0)--(v4);
\draw[blue!75!ink,line width=.6pt] (v1)--(v2);
\draw[blue!75!ink,line width=.6pt] (v1)--(v4);
\draw[blue!75!ink,line width=.6pt] (v2)--(v3);
\draw[blue!75!ink,line width=.6pt] (v2)--(v4);
\draw[blue!75!ink,line width=.6pt] (v3)--(v4);
\draw[ink,line width=.9pt] (v0)--(v1)--(v2)--(v3)--cycle;
\fill[blue] (v0) circle(1.45pt);
\node[below left,inner sep=1.3pt,text=ink] at (v0) {$ 0 $};
\fill[blue] (v1) circle(1.45pt);
\node[below right,inner sep=1.3pt,text=ink] at (v1) {$ 1 $};
\fill[blue] (v2) circle(1.45pt);
\node[above right,inner sep=1.3pt,text=ink] at (v2) {$ 2 $};
\fill[blue] (v3) circle(1.45pt);
\node[above left,inner sep=1.3pt,text=ink] at (v3) {$ 3 $};
\fill[blue] (v4) circle(1.45pt);
\node[above,inner sep=1.3pt,text=ink] at (v4) {$ 4 $};
\path[use as bounding box] (-.45,-.45) rectangle (6.45,6.45000000);
\end{tikzpicture}
\caption{$g_v=4,\quad\dim S_7^2=L_7^2=66$}\label{fig:u7local-2}
\end{subfigure}

\caption{Local direction changes on the same four-triangle fan. Both realizations attain their actual lower bounds; the different local direction types exclude this dimension difference from $\JJ$.}\label{fig:u7local}
\end{figure}
Example~\ref{ex:u7local} explains why local data must be fixed. The only interior vertex changes from two directions to four, so $(N_2,N_3)$ changes from $(1,0)$ to $(0,0)$. The entire dimension difference of three comes from $3N_2+N_3$. Both defects $\delta_7^L$ are zero, so this is not $\JJ$: each realization still satisfies its own formula exactly. Studying $\JJ$ requires equal local direction types vertex by vertex, excluding changes of local contributions.
\FloatBarrier

\section{Generic phases and their threshold}\label{sec:generic}
Generic behavior is considered for each fixed combinatorial triangulation as its vertex coordinates vary. We first give configurations with generic excess above the lower bound and then establish attainment in degree five and higher.

\subsection{Generic failure in low degrees}\label{sec:generic-failure}
The following cubic and quartic examples illustrate the chord and boundary-pair constructions discussed by \citet[Section~7]{DiPasqualeVillamizar21}. We determine their generic dimensions and give nonzero coordinate polynomials on which the strict inequalities hold.

\begin{theorem}[Cubic and quartic generic failure]\label{thm:gf}
Relative to Schumaker's lower bound with actual local corrections, $(3,2),(4,2)\in\GF$.
\end{theorem}
\begin{proof}
Examples~\ref{ex:gf3} and~\ref{ex:gf4} give fixed triangulations and nonzero coordinate polynomials whose nonvanishing sets have strictly positive defects. The conclusion follows from Theorem~\ref{thm:generic-failure-criterion}.
\end{proof}
Examples~\ref{ex:gf3} and~\ref{ex:gf4} witness these assertions. Both generic dimensions exceed the respective $L_d^2$. We give configurations, realizations, dimensions, and explicit witness polynomials. For vertices $v,i,j$, write $[vij]=(x_i-x_v)(y_j-y_v)-(y_i-y_v)(x_j-x_v)$. For a vertex-index set $A$, let $D_v(A)=\prod_{i<j,\ i,j\in A}[vij]$. The condition $D_v(A)\ne0$ says that the lines from $p_v$ to these vertices are pairwise distinct. For an edge $e=(i,j)$, let $\ell_e=0$ be a nonzero affine equation of its supporting line; an edge subscript $ij$ is shorthand for $(i,j)$.

\begin{example}[A seven-vertex cubic generic-excess configuration]\label{ex:gf3}
The boundary cycle is $(0,6,1,2)$ and the interior vertices are $3,4,5$. The eight faces are
\begin{equation}\label{eq:gf3faces}
\begin{gathered}
(3,4,5),(0,1,3),(1,2,4),(2,0,5),\\
(0,3,5),(1,4,3),(2,5,4),(1,0,6).
\end{gathered}
\end{equation}
The interior edge $(0,1)$ is a chord: both endpoints lie on the boundary.
Choose the realization
\[
\begin{aligned}
p_0&=(50,0),&p_1&=(0,50),&p_2&=(0,0),\\
p_3&=(21,20),&p_4&=(10,20),&p_5&=(20,10),\\
p_6&=(50,50).
\end{aligned}
\]
Here $(I,E,b)=(3,10,4)$; each interior vertex has four supporting lines, so $N_2=N_3=0$. The dimensions are
\[
\dim S_3^2=11,\qquad L_3^2=8.
\]
Writing $\Delta_3$ for this triangulation, an explicit witness polynomial is
\[
H_{\Delta_3,3}(p)=D_3(\{0,1,4,5\})D_4(\{1,2,3,5\})D_5(\{0,2,3,4\}).
\]
It is nonzero at these coordinates. Every realization with $H_{\Delta_3,3}(p)\ne0$ has the same dimension and lower bound, and hence witnesses generic cubic failure.
See Figure~\ref{fig:gf}\subref{fig:gf3}.
\end{example}
\begin{proof}
Suppose $H_{\Delta_3,3}(p)\ne0$. Cubic $C^2$ matching assigns a constant cofactor $q_e$ to each interior edge, with adjacent pieces differing by $\ell_e^3q_e$. At an interior vertex the compatibility equation is $\sum_{e\ni v}\epsilon_{ve}\ell_e^3q_e=0$, where $\epsilon_{ve}\in\{1,-1\}$ are coherent vertex-edge incidence signs. The four incident line forms have distinct directions. Their cubes are linearly independent in the four-dimensional space of binary cubics, as follows from the Vandermonde determinant after choosing a local coordinate chart. The compatibility equation therefore forces all four incident cofactors to vanish. The three interior vertices together meet all nine nonchord interior edges, leaving only $q_{01}$ free, since the chord has two boundary endpoints.

This remaining freedom is realized by the spline $s$ that equals $\ell_{01}^3$ on $(1,0,6)$ and zero on all other faces. Its jumps are divisible by $\ell_{01}^3$, so it is $C^2$. Every spline is the sum of a global cubic and a scalar multiple of $s$, and $s$ is not a global polynomial. Thus $\dim S_3^2=P_3+1=11$. Substituting $I=3$, $b=4$, and $N_2=N_3=0$ into Table~\ref{tab:lower} gives $L_3^2=8$. Since the displayed realization makes $H_{\Delta_3,3}$ nonzero, this proves $(3,2)\in\GF$.
\end{proof}
\begin{figure}[!htb]\centering
\begin{subfigure}[t]{.47\linewidth}\centering
\begin{tikzpicture}[scale=.78,line cap=round,every node/.style={font=\footnotesize}]
\coordinate (v0) at (6,0);
\coordinate (v1) at (0,6);
\coordinate (v2) at (0,0);
\coordinate (v3) at (2.52,2.4);
\coordinate (v4) at (1.2,2.4);
\coordinate (v5) at (2.4,1.2);
\coordinate (v6) at (6,6);
\draw[blue!75!ink,line width=.6pt] (v0)--(v3);
\draw[blue!75!ink,line width=.6pt] (v0)--(v5);
\draw[blue!75!ink,line width=.6pt] (v1)--(v3);
\draw[blue!75!ink,line width=.6pt] (v1)--(v4);
\draw[blue!75!ink,line width=.6pt] (v2)--(v4);
\draw[blue!75!ink,line width=.6pt] (v2)--(v5);
\draw[blue!75!ink,line width=.6pt] (v3)--(v4);
\draw[blue!75!ink,line width=.6pt] (v3)--(v5);
\draw[blue!75!ink,line width=.6pt] (v4)--(v5);
\draw[blue!75!ink,line width=.6pt] (v0)--(v1);
\draw[ink,line width=.9pt] (v0)--(v6)--(v1)--(v2)--cycle;
\fill[blue] (v0) circle(1.45pt);
\node[below right,inner sep=1.3pt,text=ink] at (v0) {$0$};
\fill[blue] (v1) circle(1.45pt);
\node[above,inner sep=1.3pt,text=ink] at (v1) {$1$};
\fill[blue] (v2) circle(1.45pt);
\node[below left,inner sep=1.3pt,text=ink] at (v2) {$2$};
\fill[blue] (v3) circle(1.45pt);
\node[above right,inner sep=1.3pt,text=ink] at (v3) {$3$};
\fill[blue] (v4) circle(1.45pt);
\node[above left,inner sep=1.3pt,text=ink] at (v4) {$4$};
\fill[blue] (v5) circle(1.45pt);
\node[below right,inner sep=1.3pt,text=ink] at (v5) {$5$};
\fill[blue] (v6) circle(1.45pt);
\node[above right,inner sep=1.3pt,text=ink] at (v6) {$6$};
\path[use as bounding box] (-.45,-.45) rectangle (6.45,6.45);
\end{tikzpicture}
\caption{Cubic: $\dim S_3^2=11,\quad L_3^2=8$}\label{fig:gf3}
\end{subfigure}\hfill
\begin{subfigure}[t]{.47\linewidth}\centering
\begin{tikzpicture}[scale=.78,line cap=round,every node/.style={font=\footnotesize}]
\coordinate (v0) at (3.0,6.0);
\coordinate (v1) at (0.0,0.0);
\coordinate (v2) at (6.0,0.0);
\coordinate (v3) at (2.8,0.8);
\coordinate (v4) at (3.0,3.4);
\coordinate (v5) at (2.2,1.4);
\coordinate (v6) at (3.8,1.4);
\draw[blue!75!ink,line width=.6pt] (v0)--(v1);
\draw[blue!75!ink,line width=.6pt] (v0)--(v2);
\draw[blue!75!ink,line width=.6pt] (v0)--(v4);
\draw[blue!75!ink,line width=.6pt] (v1)--(v2);
\draw[blue!75!ink,line width=.6pt] (v1)--(v3);
\draw[blue!75!ink,line width=.6pt] (v1)--(v4);
\draw[blue!75!ink,line width=.6pt] (v1)--(v5);
\draw[blue!75!ink,line width=.6pt] (v2)--(v3);
\draw[blue!75!ink,line width=.6pt] (v2)--(v4);
\draw[blue!75!ink,line width=.6pt] (v2)--(v6);
\draw[blue!75!ink,line width=.6pt] (v3)--(v5);
\draw[blue!75!ink,line width=.6pt] (v3)--(v6);
\draw[blue!75!ink,line width=.6pt] (v4)--(v5);
\draw[blue!75!ink,line width=.6pt] (v4)--(v6);
\draw[blue!75!ink,line width=.6pt] (v5)--(v6);
\draw[ink,line width=.9pt] (v0)--(v1)--(v2)--cycle;
\fill[blue] (v0) circle(1.45pt);
\node[above,inner sep=1.3pt,text=ink] at (v0) {$ 0 $};
\fill[blue] (v1) circle(1.45pt);
\node[below left,inner sep=1.3pt,text=ink] at (v1) {$ 1 $};
\fill[blue] (v2) circle(1.45pt);
\node[below right,inner sep=1.3pt,text=ink] at (v2) {$ 2 $};
\fill[blue] (v3) circle(1.45pt);
\node[below,inner sep=1.3pt,text=ink] at (v3) {$ 3 $};
\fill[blue] (v4) circle(1.45pt);
\node[above right,inner sep=1.3pt,text=ink] at (v4) {$ 4 $};
\fill[blue] (v5) circle(1.45pt);
\node[above left,inner sep=1.3pt,text=ink] at (v5) {$ 5 $};
\fill[blue] (v6) circle(1.45pt);
\node[above right,inner sep=1.3pt,text=ink] at (v6) {$ 6 $};
\path[use as bounding box] (-.45,-.45) rectangle (6.45,6.45000000);
\end{tikzpicture}
\caption{Quartic: $\dim S_4^2=16,\quad L_4^2=15$}\label{fig:gf4}
\end{subfigure}
\caption{Cubic and quartic generic-excess configurations from Examples~\ref{ex:gf3} and~\ref{ex:gf4}. Both generic dimensions strictly exceed their actual lower bounds.}\label{fig:gf}
\end{figure}
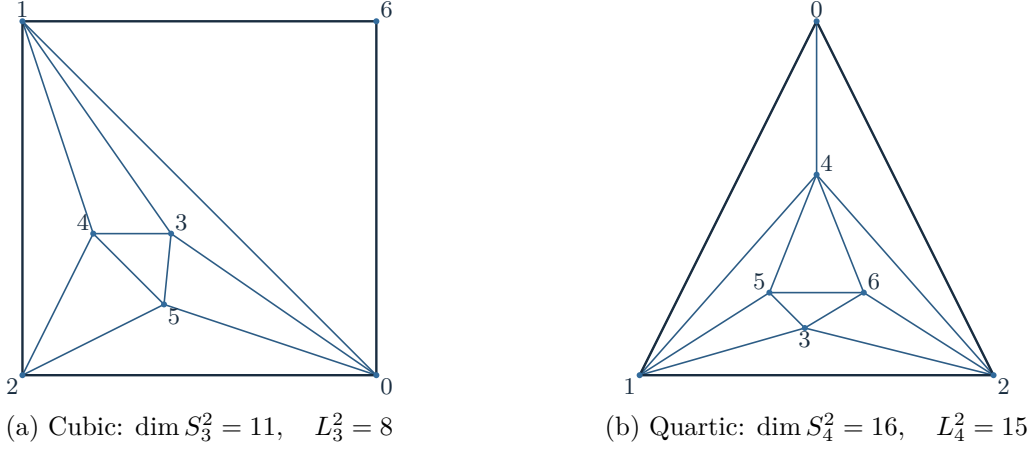

\begin{remark}[Relation to earlier low-degree constructions]\label{rem:polynomial-baseline}
\citet[Section~7]{DiPasqualeVillamizar21} discuss the chord and boundary-pair constructions in connection with generic equality to the polynomial-floor-adjusted bound
\[
B_d^2=\max\{P_d,L_d^2\}.
\]
Examples~\ref{ex:gf3} and~\ref{ex:gf4} provide such witnesses and determine their generic dimensions. The cubic example satisfies $L_3^2=8<P_3=10<\dim S_3^2=11$, and the quartic example satisfies $L_4^2=P_4=15<\dim S_4^2=16$. Thus both constructions also exhibit strict excess over $B_d^2$. In the cubic example, two of the three dimensions above $L_3^2$ are forced by $P_3$, and one remains after this floor is included.

\end{remark}

\begin{example}[A seven-vertex quartic generic-excess configuration]\label{ex:gf4}
The boundary vertices are $0,1,2$ and the interior vertices are $3,4,5,6$. The nine faces are
\[
\begin{gathered}
(1,4,5),(6,5,3),(6,5,4),(2,6,3),(0,4,1),\\
(1,3,2),(1,3,5),(2,4,0),(2,4,6).
\end{gathered}
\]
The faces $(0,4,1)$ and $(2,4,0)$ form a boundary pair: their common edge is $(0,4)$, and $(0,1)$ and $(0,2)$ are boundary edges.
Choose coordinates
\[
\begin{aligned}
p_0&=(15,30),&p_1&=(0,0),&p_2&=(30,0),\\
p_3&=(14,4),&p_4&=(15,17),&p_5&=(11,7),&p_6&=(19,7).
\end{aligned}
\]
Here $(I,E,b)=(4,12,3)$ and the interior valences are $(4,5,4,4)$. The dimension of this realization, also the generic dimension of this combinatorial type, is
\[
\dim S_4^2=16,\qquad L_4^2=15.
\]
Writing $\Delta_4$ for this triangulation, an explicit witness polynomial is
\[
\begin{aligned}
H_{\Delta_4,4}(p)={}&D_3(\{1,2,5,6\})D_4(\{0,1,2,5,6\})\\
&\cdot D_5(\{1,3,4,6\})D_6(\{2,3,4,5\}).
\end{aligned}
\]
It is nonzero at these coordinates. Every realization with $H_{\Delta_4,4}(p)\ne0$ satisfies
\[
\dim S_4^2\ge16>15=L_4^2.
\]
This strict inequality suffices to witness generic quartic failure.
See Figure~\ref{fig:gf}\subref{fig:gf4}.
\end{example}
\begin{proof}
For ordered vertices $i,j$, take $\ell_{ij}(z)=\det(p_j-p_i,z-p_i)$, and orient each edge from its smaller to its larger index. The boundary-pair construction can be written explicitly. Choose nonzero scalars $a,b$ so that $X=a\ell_{14}$, $Y=b\ell_{24}$, and $\ell_{04}=X+Y$. Define
\[
s|_{(0,4,1)}=X^3(X+2Y),\qquad
s|_{(2,4,0)}=Y^3(2X+Y),
\]
and set $s=0$ on the other faces. Across $(1,4)$ and $(2,4)$, the jumps have the required cubic line factors. Across $(0,4)$, use
\[
X^3(X+2Y)-Y^3(2X+Y)=(X-Y)(X+Y)^3.
\]
Thus $s\in S_4^2$, and it is not a global polynomial because it is nonzero but vanishes on another face. Consequently $\dim S_4^2\ge16$. On $H_{\Delta_4,4}\ne0$, Table~\ref{tab:lower} gives $L_4^2=15$, proving $(4,2)\in\GF$. A further calculation gives the exact generic dimension $16$.
\end{proof}

\FloatBarrier

\subsection{Generic quintic attainment and the full generic phase}\label{sec:generic-attainment}
The behavior changes in degree five: every fixed combinatorial triangulation generically attains the bound. The quintic and sextic results in Appendix~\ref{app:g5}, together with the universal range in Section~\ref{sec:universal}, determine the full generic-attainment range.

\begin{theorem}[Generic quintic attainment]\label{thm:generic}
For every fixed finite abstract disk triangulation $\Delta$, there is a nonzero coordinate polynomial $H_{\Delta,5}$ such that every $p\in\mathcal R_\Delta$ with $H_{\Delta,5}(p)\ne0$ satisfies
\[
\dim S_5^2(\Delta,p)=L_5^2(\Delta,p).
\]
\end{theorem}
The proof is in Appendix~\ref{app:g5}. For each combinatorial triangulation, it constructs a full-rank realization, retaining one quadratic cofactor per physical edge. Endpoint enhancement, local surjectivity, and geometric induction establish the all-triangulation statement.

\begin{theorem}[The generic phase and its threshold]\label{thm:gthreshold}
On the $C^2$ slice with $d\ge3$,
\[
\GF=\{(3,2),(4,2)\},\qquad
\GG=\{(d,2):d\ge5\}.
\]
Here $g(2)$ is the least $d_0\ge3$ such that $(d,2)\in\GG$ for every $d\ge d_0$; hence $g(2)=5$. On the same generic realizations as in Theorem~\ref{thm:generic}, the bound is attained for all $d\ge5$.
\end{theorem}
\begin{proof}
Appendix~\ref{app:g5} gives attainment in degrees five and six under the same polynomial condition. Theorem~\ref{thm:universal} and Hong's theorem \citep{Hong91} cover all $d\ge7$, while Examples~\ref{ex:gf3} and~\ref{ex:gf4} exclude generic attainment in degrees three and four.
\end{proof}
\noindent The generic phase describes dimension at general coordinates; it does not exclude the special geometric jumps in the next section.

\section{Geometric jumps with fixed local data}\label{sec:jumps}
We compare pairs of realizations with the same combinatorics and vertexwise local direction types. Their actual lower bounds agree, so different dimensions reflect different global defects.

\subsection{Results and data conventions}
We first summarize the jump results and specify the conventions used to present the realizing pairs.

The \emph{full orientation-sign data} are the signs of the determinants $[ijk]$ for all ordered triples of distinct vertices, including zero signs; the determinant $[ijk]$ was defined in Section~\ref{sec:generic-failure}.

\begin{theorem}[Geometric jumps in degrees three to five]\label{thm:jumps}
The pairs $(3,2),(4,2),(5,2)$ all belong to $\JJ$. Quartic and quintic jumps occur even between realizations with identical full orientation-sign data.
\end{theorem}
Examples~\ref{ex:j3}, \ref{ex:j4}, and~\ref{ex:j5} give the respective jump witnesses. Each pair has the \textbf{same triangle list} and equal local direction types vertex by vertex. All coordinates are exact rationals. We call the two members Realization~1 and Realization~2. Table~\ref{tab:jumps} lists the exact ranks of the cofactor compatibility matrices $C_d$ and the dimensions obtained from \eqref{eq:polynomial-constraint-rank}.

\subsection{Cubics: a ten-triangle hexagon}\label{sec:jump-cubic}
The first pair moves a single interior vertex of a hexagonal triangulation. Both realizations retain five distinct supporting directions at every interior vertex.

\begin{example}[A cubic jump with fixed local data]\label{ex:j3}
The boundary hexagon has coordinates
\[
\begin{aligned}
p_0&=(40,0),&p_1&=(0,40),&p_2&=(-40,40),\\
p_3&=(-40,0),&p_4&=(0,-40),&p_5&=(40,-40).
\end{aligned}
\]
The two interior coordinate sets are
\[
\begin{array}{c|ccc}
&p_6&p_7&p_8\\\hline
\text{Realization 1}&(1,10)&(-10,0)&(10,-10)\\
\text{Realization 2}&(0,10)&(-10,0)&(10,-10)
\end{array}
\]
The ten faces are
\[
\begin{gathered}
(6,7,8),(6,0,1),(6,1,2),(7,2,3),(7,3,4),\\
(8,4,5),(8,5,0),(6,7,2),(7,8,4),(8,6,0).
\end{gathered}
\]
Both realizations have $(I,E,b)=(3,12,6)$ and five distinct supporting lines at every interior vertex. Their dimensions are
\[
L_3^2=10,\qquad
\dim S_3^2(\text{Realization 1})=10,\qquad
\dim S_3^2(\text{Realization 2})=11.
\]
Realizations~1 and~2 appear in Figures~\ref{fig:j3}\subref{fig:j3-1} and~\ref{fig:j3}\subref{fig:j3-2}.
\end{example}
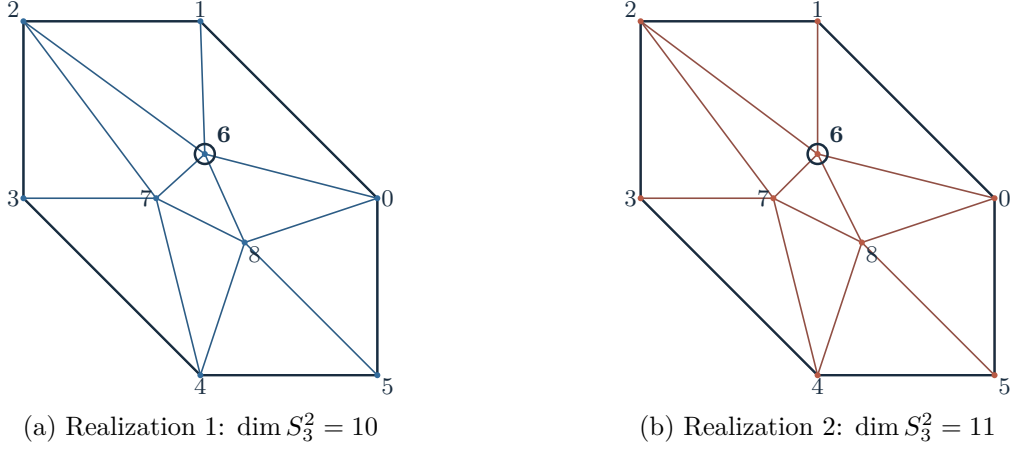
\begin{figure}[!htbp]\centering
\begin{subfigure}[t]{.47\linewidth}\centering
\begin{tikzpicture}[scale=.78,line cap=round,every node/.style={font=\footnotesize}]
\coordinate (v0) at (6.000000000,3.000000000);
\coordinate (v1) at (3.000000000,6.000000000);
\coordinate (v2) at (0.000000000,6.000000000);
\coordinate (v3) at (0.000000000,3.000000000);
\coordinate (v4) at (3.000000000,0.000000000);
\coordinate (v5) at (6.000000000,0.000000000);
\coordinate (v6) at (3.075000000,3.750000000);
\coordinate (v7) at (2.250000000,3.000000000);
\coordinate (v8) at (3.750000000,2.250000000);
\draw[blue!75!ink,line width=.6pt] (v0)--(v1);
\draw[blue!75!ink,line width=.6pt] (v0)--(v5);
\draw[blue!75!ink,line width=.6pt] (v0)--(v6);
\draw[blue!75!ink,line width=.6pt] (v0)--(v8);
\draw[blue!75!ink,line width=.6pt] (v1)--(v2);
\draw[blue!75!ink,line width=.6pt] (v1)--(v6);
\draw[blue!75!ink,line width=.6pt] (v2)--(v3);
\draw[blue!75!ink,line width=.6pt] (v2)--(v6);
\draw[blue!75!ink,line width=.6pt] (v2)--(v7);
\draw[blue!75!ink,line width=.6pt] (v3)--(v4);
\draw[blue!75!ink,line width=.6pt] (v3)--(v7);
\draw[blue!75!ink,line width=.6pt] (v4)--(v5);
\draw[blue!75!ink,line width=.6pt] (v4)--(v7);
\draw[blue!75!ink,line width=.6pt] (v4)--(v8);
\draw[blue!75!ink,line width=.6pt] (v5)--(v8);
\draw[blue!75!ink,line width=.6pt] (v6)--(v7);
\draw[blue!75!ink,line width=.6pt] (v6)--(v8);
\draw[blue!75!ink,line width=.6pt] (v7)--(v8);
\draw[ink,line width=.9pt] (v0)--(v1)--(v2)--(v3)--(v4)--(v5)--cycle;
\fill[blue] (v0) circle(1.45pt);
\node[right,inner sep=1.3pt,text=ink] at (v0) {$ 0 $};
\fill[blue] (v1) circle(1.45pt);
\node[above,inner sep=1.3pt,text=ink] at (v1) {$ 1 $};
\fill[blue] (v2) circle(1.45pt);
\node[above left,inner sep=1.3pt,text=ink] at (v2) {$ 2 $};
\fill[blue] (v3) circle(1.45pt);
\node[left,inner sep=1.3pt,text=ink] at (v3) {$ 3 $};
\fill[blue] (v4) circle(1.45pt);
\node[below,inner sep=1.3pt,text=ink] at (v4) {$ 4 $};
\fill[blue] (v5) circle(1.45pt);
\node[below right,inner sep=1.3pt,text=ink] at (v5) {$ 5 $};
\fill[blue] (v6) circle(1.45pt);
\draw[ink,line width=.95pt] (v6) circle(4.8pt);
\node[above right,xshift=3pt,yshift=3pt,inner sep=1.3pt,text=ink] at (v6) {$\mathbf{6}$};
\fill[blue] (v7) circle(1.45pt);
\node[left,inner sep=1.3pt,text=ink] at (v7) {$ 7 $};
\fill[blue] (v8) circle(1.45pt);
\node[below right,inner sep=1.3pt,text=ink] at (v8) {$ 8 $};
\path[use as bounding box] (-.45,-.45) rectangle (6.45,6.45000000);
\end{tikzpicture}
\caption{Realization~1: $\dim S_3^2=10$}\label{fig:j3-1}
\end{subfigure}\hfill
\begin{subfigure}[t]{.47\linewidth}\centering
\begin{tikzpicture}[scale=.78,line cap=round,every node/.style={font=\footnotesize}]
\coordinate (v0) at (6.000000000,3.000000000);
\coordinate (v1) at (3.000000000,6.000000000);
\coordinate (v2) at (0.000000000,6.000000000);
\coordinate (v3) at (0.000000000,3.000000000);
\coordinate (v4) at (3.000000000,0.000000000);
\coordinate (v5) at (6.000000000,0.000000000);
\coordinate (v6) at (3.000000000,3.750000000);
\coordinate (v7) at (2.250000000,3.000000000);
\coordinate (v8) at (3.750000000,2.250000000);
\draw[brick!75!ink,line width=.6pt] (v0)--(v1);
\draw[brick!75!ink,line width=.6pt] (v0)--(v5);
\draw[brick!75!ink,line width=.6pt] (v0)--(v6);
\draw[brick!75!ink,line width=.6pt] (v0)--(v8);
\draw[brick!75!ink,line width=.6pt] (v1)--(v2);
\draw[brick!75!ink,line width=.6pt] (v1)--(v6);
\draw[brick!75!ink,line width=.6pt] (v2)--(v3);
\draw[brick!75!ink,line width=.6pt] (v2)--(v6);
\draw[brick!75!ink,line width=.6pt] (v2)--(v7);
\draw[brick!75!ink,line width=.6pt] (v3)--(v4);
\draw[brick!75!ink,line width=.6pt] (v3)--(v7);
\draw[brick!75!ink,line width=.6pt] (v4)--(v5);
\draw[brick!75!ink,line width=.6pt] (v4)--(v7);
\draw[brick!75!ink,line width=.6pt] (v4)--(v8);
\draw[brick!75!ink,line width=.6pt] (v5)--(v8);
\draw[brick!75!ink,line width=.6pt] (v6)--(v7);
\draw[brick!75!ink,line width=.6pt] (v6)--(v8);
\draw[brick!75!ink,line width=.6pt] (v7)--(v8);
\draw[ink,line width=.9pt] (v0)--(v1)--(v2)--(v3)--(v4)--(v5)--cycle;
\fill[brick] (v0) circle(1.45pt);
\node[right,inner sep=1.3pt,text=ink] at (v0) {$ 0 $};
\fill[brick] (v1) circle(1.45pt);
\node[above,inner sep=1.3pt,text=ink] at (v1) {$ 1 $};
\fill[brick] (v2) circle(1.45pt);
\node[above left,inner sep=1.3pt,text=ink] at (v2) {$ 2 $};
\fill[brick] (v3) circle(1.45pt);
\node[left,inner sep=1.3pt,text=ink] at (v3) {$ 3 $};
\fill[brick] (v4) circle(1.45pt);
\node[below,inner sep=1.3pt,text=ink] at (v4) {$ 4 $};
\fill[brick] (v5) circle(1.45pt);
\node[below right,inner sep=1.3pt,text=ink] at (v5) {$ 5 $};
\fill[brick] (v6) circle(1.45pt);
\draw[ink,line width=.95pt] (v6) circle(4.8pt);
\node[above right,xshift=3pt,yshift=3pt,inner sep=1.3pt,text=ink] at (v6) {$\mathbf{6}$};
\fill[brick] (v7) circle(1.45pt);
\node[left,inner sep=1.3pt,text=ink] at (v7) {$ 7 $};
\fill[brick] (v8) circle(1.45pt);
\node[below right,inner sep=1.3pt,text=ink] at (v8) {$ 8 $};
\path[use as bounding box] (-.45,-.45) rectangle (6.45,6.45000000);
\end{tikzpicture}
\caption{Realization~2: $\dim S_3^2=11$}\label{fig:j3-2}
\end{subfigure}

\caption{The two cubic-jump realizations. Circles mark the only moving vertex, $p_6$, which changes from $(1,10)$ to $(0,10)$. The triangle incidences and local direction types remain fixed.}\label{fig:j3}
\end{figure}

\FloatBarrier

\subsection{Quartics: a jump within one orientation-sign chamber}

The geometry dependence of quartic $C^2$ splines on the Morgan--Scott triangulation was studied by \citet{Diener90} and \citet*{DengFengKozak00}. The following pair exhibits this dimension change while preserving the full orientation-sign data.

\begin{example}[A quartic Morgan--Scott jump]\label{ex:j4}
Take the seven-face Morgan--Scott configuration, with boundary vertices $0,1,2$ and interior vertices $3,4,5$. Its faces are
\begin{equation}\label{eq:msfaces}
\begin{gathered}
(3,4,5),(0,1,3),(1,2,4),(2,0,5),\\
(0,3,5),(1,4,3),(2,5,4).
\end{gathered}
\end{equation}
Fix boundary coordinates
\[
p_0=(50,0),\qquad p_1=(0,50),\qquad p_2=(0,0).
\]
The interior coordinates are
\[
\begin{array}{c|ccc}
&p_3&p_4&p_5\\\hline
\text{Realization 1}&(21,20)&(10,20)&(20,10)\\
\text{Realization 2}&(20,20)&(10,20)&(20,10)
\end{array}
\]
Every interior vertex retains four supporting lines. The signs of all ordered three-point determinants agree, so the full orientation-sign data are unchanged. The dimensions are
\[
L_4^2=15,\qquad
\dim S_4^2(\text{Realization 1})=15,\qquad
\dim S_4^2(\text{Realization 2})=16.
\]
See Figures~\ref{fig:j4}\subref{fig:j4-1} and~\ref{fig:j4}\subref{fig:j4-2}.
\end{example}
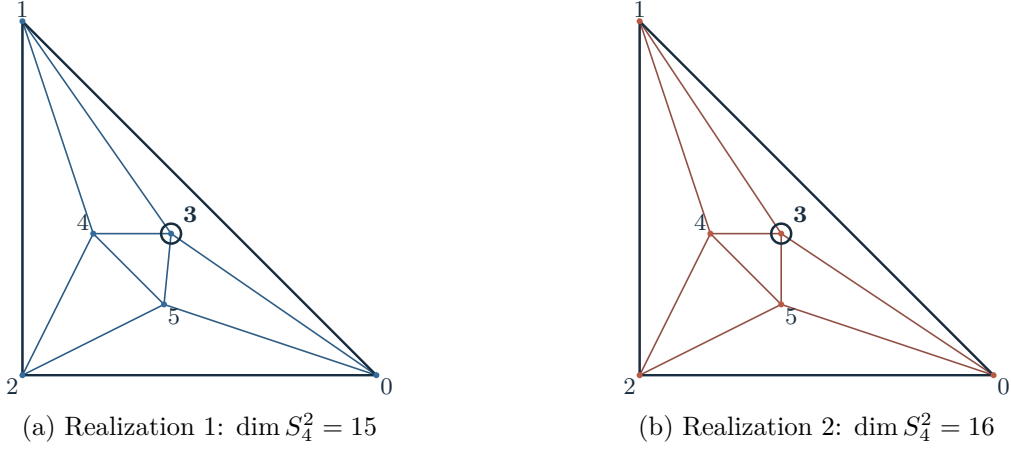
\begin{figure}[!htbp]\centering
\begin{subfigure}[t]{.47\linewidth}\centering
\begin{tikzpicture}[scale=.78,line cap=round,every node/.style={font=\footnotesize}]
\coordinate (v0) at (6.000000000,0.000000000);
\coordinate (v1) at (0.000000000,6.000000000);
\coordinate (v2) at (0.000000000,0.000000000);
\coordinate (v3) at (2.520000000,2.400000000);
\coordinate (v4) at (1.200000000,2.400000000);
\coordinate (v5) at (2.400000000,1.200000000);
\draw[blue!75!ink,line width=.6pt] (v0)--(v1);
\draw[blue!75!ink,line width=.6pt] (v0)--(v2);
\draw[blue!75!ink,line width=.6pt] (v0)--(v3);
\draw[blue!75!ink,line width=.6pt] (v0)--(v5);
\draw[blue!75!ink,line width=.6pt] (v1)--(v2);
\draw[blue!75!ink,line width=.6pt] (v1)--(v3);
\draw[blue!75!ink,line width=.6pt] (v1)--(v4);
\draw[blue!75!ink,line width=.6pt] (v2)--(v4);
\draw[blue!75!ink,line width=.6pt] (v2)--(v5);
\draw[blue!75!ink,line width=.6pt] (v3)--(v4);
\draw[blue!75!ink,line width=.6pt] (v3)--(v5);
\draw[blue!75!ink,line width=.6pt] (v4)--(v5);
\draw[ink,line width=.9pt] (v0)--(v1)--(v2)--cycle;
\fill[blue] (v0) circle(1.45pt);
\node[below right,inner sep=1.3pt,text=ink] at (v0) {$ 0 $};
\fill[blue] (v1) circle(1.45pt);
\node[above,inner sep=1.3pt,text=ink] at (v1) {$ 1 $};
\fill[blue] (v2) circle(1.45pt);
\node[below left,inner sep=1.3pt,text=ink] at (v2) {$ 2 $};
\fill[blue] (v3) circle(1.45pt);
\draw[ink,line width=.95pt] (v3) circle(4.8pt);
\node[above right,xshift=3pt,yshift=3pt,inner sep=1.3pt,text=ink] at (v3) {$\mathbf{3}$};
\fill[blue] (v4) circle(1.45pt);
\node[above left,inner sep=1.3pt,text=ink] at (v4) {$ 4 $};
\fill[blue] (v5) circle(1.45pt);
\node[below right,inner sep=1.3pt,text=ink] at (v5) {$ 5 $};
\path[use as bounding box] (-.45,-.45) rectangle (6.45,6.45000000);
\end{tikzpicture}
\caption{Realization~1: $\dim S_4^2=15$}\label{fig:j4-1}
\end{subfigure}\hfill
\begin{subfigure}[t]{.47\linewidth}\centering
\begin{tikzpicture}[scale=.78,line cap=round,every node/.style={font=\footnotesize}]
\coordinate (v0) at (6.000000000,0.000000000);
\coordinate (v1) at (0.000000000,6.000000000);
\coordinate (v2) at (0.000000000,0.000000000);
\coordinate (v3) at (2.400000000,2.400000000);
\coordinate (v4) at (1.200000000,2.400000000);
\coordinate (v5) at (2.400000000,1.200000000);
\draw[brick!75!ink,line width=.6pt] (v0)--(v1);
\draw[brick!75!ink,line width=.6pt] (v0)--(v2);
\draw[brick!75!ink,line width=.6pt] (v0)--(v3);
\draw[brick!75!ink,line width=.6pt] (v0)--(v5);
\draw[brick!75!ink,line width=.6pt] (v1)--(v2);
\draw[brick!75!ink,line width=.6pt] (v1)--(v3);
\draw[brick!75!ink,line width=.6pt] (v1)--(v4);
\draw[brick!75!ink,line width=.6pt] (v2)--(v4);
\draw[brick!75!ink,line width=.6pt] (v2)--(v5);
\draw[brick!75!ink,line width=.6pt] (v3)--(v4);
\draw[brick!75!ink,line width=.6pt] (v3)--(v5);
\draw[brick!75!ink,line width=.6pt] (v4)--(v5);
\draw[ink,line width=.9pt] (v0)--(v1)--(v2)--cycle;
\fill[brick] (v0) circle(1.45pt);
\node[below right,inner sep=1.3pt,text=ink] at (v0) {$ 0 $};
\fill[brick] (v1) circle(1.45pt);
\node[above,inner sep=1.3pt,text=ink] at (v1) {$ 1 $};
\fill[brick] (v2) circle(1.45pt);
\node[below left,inner sep=1.3pt,text=ink] at (v2) {$ 2 $};
\fill[brick] (v3) circle(1.45pt);
\draw[ink,line width=.95pt] (v3) circle(4.8pt);
\node[above right,xshift=3pt,yshift=3pt,inner sep=1.3pt,text=ink] at (v3) {$\mathbf{3}$};
\fill[brick] (v4) circle(1.45pt);
\node[above left,inner sep=1.3pt,text=ink] at (v4) {$ 4 $};
\fill[brick] (v5) circle(1.45pt);
\node[below right,inner sep=1.3pt,text=ink] at (v5) {$ 5 $};
\path[use as bounding box] (-.45,-.45) rectangle (6.45,6.45000000);
\end{tikzpicture}
\caption{Realization~2: $\dim S_4^2=16$}\label{fig:j4-2}
\end{subfigure}

\caption{Quartic-jump Realizations~1 and~2. Circles mark the only moving vertex, $p_3$, changing from $(21,20)$ to $(20,20)$. The full orientation-sign data and vertexwise local direction types agree; the dimensions are $15$ and $16$.}\label{fig:j4}
\end{figure}

\FloatBarrier

\subsection{Quintics: generic attainment coexisting with jumps}\label{sec:jump-quintic}
The following pair shows that generic quintic attainment can coexist with a geometric jump on the same combinatorial triangulation.

\begin{example}[A nine-vertex three-sun configuration]\label{ex:j5}
The boundary vertices are $0,1,2$ and the interior vertices are $3,4,5,6,7,8$. The thirteen faces are
\[
\begin{gathered}
(0,1,8),(0,2,7),(0,5,7),(0,5,8),(1,2,6),\\
(1,4,6),(1,4,8),(2,3,6),(2,3,7),(3,4,5),\\
(3,4,6),(3,5,7),(4,5,8).
\end{gathered}
\]
Fix the boundary coordinates
\[
p_0=(0,0),\qquad p_1=(30,0),\qquad p_2=(15,26),
\]
and the interior coordinates
\[
p_3=(15,10),\qquad p_4=(210/13,8),\qquad p_5=(180/13,8).
\]
The remaining interior coordinates are
\[
\begin{array}{c|ccc}
&p_6&p_7&p_8\\\hline
\rule{0pt}{4.5ex}\text{Realization 1}
&\left(\dfrac{1500}{83},\dfrac{2600}{249}\right)
&\left(\dfrac{990}{83},\dfrac{2600}{249}\right)
&\left(15,\dfrac{1274}{249}\right)\\[4pt]
\text{Realization 2}
&\left(18,\dfrac{52}{5}\right)
&\left(12,\dfrac{52}{5}\right)
&\left(15,\dfrac{26}{5}\right)
\end{array}
\]
Only vertices $6,7,8$ move; all other vertices are fixed. Both realizations have the same full orientation-sign data and $(I,E,b)=(6,18,3)$. Vertices $3,4,5$ have valence five, and vertices $6,7,8$ have valence four, with pairwise distinct supporting directions at every interior vertex. The dimensions are
\[
L_5^2=39,\qquad
\dim S_5^2(\text{Realization 1})=39,\qquad
\dim S_5^2(\text{Realization 2})=40.
\]
See Figures~\ref{fig:j5}\subref{fig:j5-1} and~\ref{fig:j5}\subref{fig:j5-2}.
\end{example}
\begin{figure}[!htbp]\centering
\begin{subfigure}[t]{.47\linewidth}\centering
\begin{tikzpicture}[scale=.78,line cap=round,every node/.style={font=\footnotesize}]
\coordinate (v0) at (0.000000000,0.000000000);
\coordinate (v1) at (6.000000000,0.000000000);
\coordinate (v2) at (3.000000000,5.200000000);
\coordinate (v3) at (3.000000000,2.000000000);
\coordinate (v4) at (3.230769231,1.600000000);
\coordinate (v5) at (2.769230769,1.600000000);
\coordinate (v6) at (3.614457831,2.088353414);
\coordinate (v7) at (2.385542169,2.088353414);
\coordinate (v8) at (3.000000000,1.023293173);
\draw[blue!75!ink,line width=.6pt] (v0)--(v1);
\draw[blue!75!ink,line width=.6pt] (v0)--(v2);
\draw[blue!75!ink,line width=.6pt] (v0)--(v5);
\draw[blue!75!ink,line width=.6pt] (v0)--(v7);
\draw[blue!75!ink,line width=.6pt] (v0)--(v8);
\draw[blue!75!ink,line width=.6pt] (v1)--(v2);
\draw[blue!75!ink,line width=.6pt] (v1)--(v4);
\draw[blue!75!ink,line width=.6pt] (v1)--(v6);
\draw[blue!75!ink,line width=.6pt] (v1)--(v8);
\draw[blue!75!ink,line width=.6pt] (v2)--(v3);
\draw[blue!75!ink,line width=.6pt] (v2)--(v6);
\draw[blue!75!ink,line width=.6pt] (v2)--(v7);
\draw[blue!75!ink,line width=.6pt] (v3)--(v4);
\draw[blue!75!ink,line width=.6pt] (v3)--(v5);
\draw[blue!75!ink,line width=.6pt] (v3)--(v6);
\draw[blue!75!ink,line width=.6pt] (v3)--(v7);
\draw[blue!75!ink,line width=.6pt] (v4)--(v5);
\draw[blue!75!ink,line width=.6pt] (v4)--(v6);
\draw[blue!75!ink,line width=.6pt] (v4)--(v8);
\draw[blue!75!ink,line width=.6pt] (v5)--(v7);
\draw[blue!75!ink,line width=.6pt] (v5)--(v8);
\draw[ink,line width=.9pt] (v0)--(v1)--(v2)--cycle;
\fill[blue] (v0) circle(1.45pt);
\node[below left,inner sep=1.3pt,text=ink] at (v0) {$ 0 $};
\fill[blue] (v1) circle(1.45pt);
\node[below right,inner sep=1.3pt,text=ink] at (v1) {$ 1 $};
\fill[blue] (v2) circle(1.45pt);
\node[above,inner sep=1.3pt,text=ink] at (v2) {$ 2 $};
\fill[blue] (v3) circle(1.45pt);
\node[above right,inner sep=1.3pt,text=ink] at (v3) {$ 3 $};
\fill[blue] (v4) circle(1.45pt);
\node[below right,inner sep=1.3pt,text=ink] at (v4) {$ 4 $};
\fill[blue] (v5) circle(1.45pt);
\node[below left,inner sep=1.3pt,text=ink] at (v5) {$ 5 $};
\fill[blue] (v6) circle(1.45pt);
\draw[ink,line width=.95pt] (v6) circle(4.8pt);
\node[above right,xshift=3pt,yshift=3pt,inner sep=1.3pt,text=ink] at (v6) {$\mathbf{6}$};
\fill[blue] (v7) circle(1.45pt);
\draw[ink,line width=.95pt] (v7) circle(4.8pt);
\node[above left,xshift=-3pt,yshift=3pt,inner sep=1.3pt,text=ink] at (v7) {$\mathbf{7}$};
\fill[blue] (v8) circle(1.45pt);
\draw[ink,line width=.95pt] (v8) circle(4.8pt);
\node[below,yshift=-4pt,inner sep=1.3pt,text=ink] at (v8) {$\mathbf{8}$};
\path[use as bounding box] (-.45,-.45) rectangle (6.45,5.65);
\end{tikzpicture}
\caption{Realization~1: $\dim S_5^2=39$}\label{fig:j5-1}
\end{subfigure}\hfill
\begin{subfigure}[t]{.47\linewidth}\centering
\begin{tikzpicture}[scale=.78,line cap=round,every node/.style={font=\footnotesize}]
\coordinate (v0) at (0.000000000,0.000000000);
\coordinate (v1) at (6.000000000,0.000000000);
\coordinate (v2) at (3.000000000,5.200000000);
\coordinate (v3) at (3.000000000,2.000000000);
\coordinate (v4) at (3.230769231,1.600000000);
\coordinate (v5) at (2.769230769,1.600000000);
\coordinate (v6) at (3.600000000,2.080000000);
\coordinate (v7) at (2.400000000,2.080000000);
\coordinate (v8) at (3.000000000,1.040000000);
\draw[brick!75!ink,line width=.6pt] (v0)--(v1);
\draw[brick!75!ink,line width=.6pt] (v0)--(v2);
\draw[brick!75!ink,line width=.6pt] (v0)--(v5);
\draw[brick!75!ink,line width=.6pt] (v0)--(v7);
\draw[brick!75!ink,line width=.6pt] (v0)--(v8);
\draw[brick!75!ink,line width=.6pt] (v1)--(v2);
\draw[brick!75!ink,line width=.6pt] (v1)--(v4);
\draw[brick!75!ink,line width=.6pt] (v1)--(v6);
\draw[brick!75!ink,line width=.6pt] (v1)--(v8);
\draw[brick!75!ink,line width=.6pt] (v2)--(v3);
\draw[brick!75!ink,line width=.6pt] (v2)--(v6);
\draw[brick!75!ink,line width=.6pt] (v2)--(v7);
\draw[brick!75!ink,line width=.6pt] (v3)--(v4);
\draw[brick!75!ink,line width=.6pt] (v3)--(v5);
\draw[brick!75!ink,line width=.6pt] (v3)--(v6);
\draw[brick!75!ink,line width=.6pt] (v3)--(v7);
\draw[brick!75!ink,line width=.6pt] (v4)--(v5);
\draw[brick!75!ink,line width=.6pt] (v4)--(v6);
\draw[brick!75!ink,line width=.6pt] (v4)--(v8);
\draw[brick!75!ink,line width=.6pt] (v5)--(v7);
\draw[brick!75!ink,line width=.6pt] (v5)--(v8);
\draw[ink,line width=.9pt] (v0)--(v1)--(v2)--cycle;
\fill[brick] (v0) circle(1.45pt);
\node[below left,inner sep=1.3pt,text=ink] at (v0) {$ 0 $};
\fill[brick] (v1) circle(1.45pt);
\node[below right,inner sep=1.3pt,text=ink] at (v1) {$ 1 $};
\fill[brick] (v2) circle(1.45pt);
\node[above,inner sep=1.3pt,text=ink] at (v2) {$ 2 $};
\fill[brick] (v3) circle(1.45pt);
\node[above right,inner sep=1.3pt,text=ink] at (v3) {$ 3 $};
\fill[brick] (v4) circle(1.45pt);
\node[below right,inner sep=1.3pt,text=ink] at (v4) {$ 4 $};
\fill[brick] (v5) circle(1.45pt);
\node[below left,inner sep=1.3pt,text=ink] at (v5) {$ 5 $};
\fill[brick] (v6) circle(1.45pt);
\draw[ink,line width=.95pt] (v6) circle(4.8pt);
\node[above right,xshift=3pt,yshift=3pt,inner sep=1.3pt,text=ink] at (v6) {$\mathbf{6}$};
\fill[brick] (v7) circle(1.45pt);
\draw[ink,line width=.95pt] (v7) circle(4.8pt);
\node[above left,xshift=-3pt,yshift=3pt,inner sep=1.3pt,text=ink] at (v7) {$\mathbf{7}$};
\fill[brick] (v8) circle(1.45pt);
\draw[ink,line width=.95pt] (v8) circle(4.8pt);
\node[below,yshift=-4pt,inner sep=1.3pt,text=ink] at (v8) {$\mathbf{8}$};
\path[use as bounding box] (-.45,-.45) rectangle (6.45,5.65);
\end{tikzpicture}
\caption{Realization~2: $\dim S_5^2=40$}\label{fig:j5-2}
\end{subfigure}

\caption{Quintic-jump Realizations~1 and~2, drawn at the same scale with their actual coordinates. Circles mark the moving vertices $p_6,p_7,p_8$; all other vertices are fixed. The full orientation-sign data and vertexwise local direction types agree; the dimensions are $39$ and $40$.}\label{fig:j5}
\end{figure}
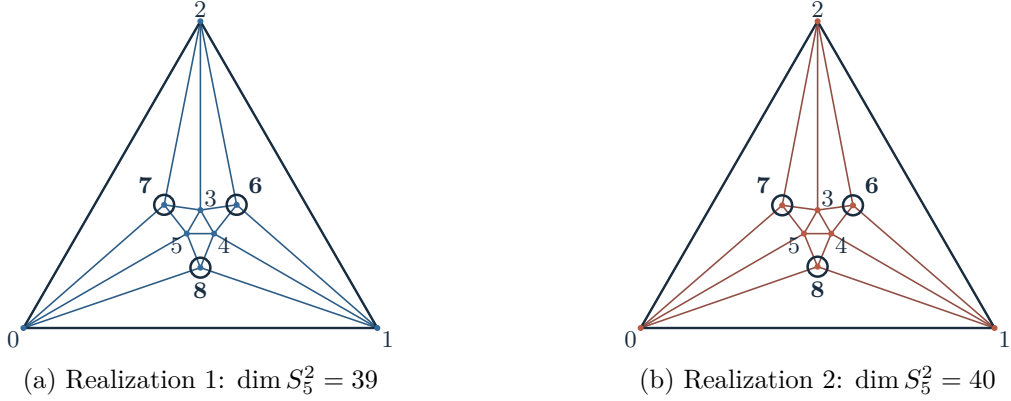

Together with Theorem~\ref{thm:generic}, this example realizes the $\GG\cap\JJ$ entry in Figure~\ref{fig:phase}. A jump with fixed local data occurs above the same lower bound, unlike a dimension change caused by a change in the lower bound itself.

\begin{table}[!htbp]\centering\small
\begin{tabular}{ccrrrr}\toprule
&&\multicolumn{2}{c}{$\rank C_d$}&\multicolumn{2}{c}{$\dim S_d^2$}\\
$d$&$L_d^2$&Realization~1&Realization~2&Realization~1&Realization~2\\\midrule
3&10&12&11&10&11\\
4&15&27&26&15&16\\
5&39&90&89&39&40\\\bottomrule
\end{tabular}
\caption{Exact cofactor ranks, common lower bounds, and dimensions of the three geometric-jump pairs.}\label{tab:jumps}
\end{table}
\FloatBarrier

\section{Open questions in degree six}\label{sec:open}
Theorem~\ref{thm:gthreshold} gives generic attainment for $S_6^2$. The unresolved entries in Figure~\ref{fig:phase} concern universal attainment and geometric jumps with fixed local data. We state these two questions and explain how they are related.

\begin{question}[Universal sextic attainment?]\label{q:s6}
For every finite straight-line triangulated disk and every realization, is
\begin{equation}\label{eq:s6open}
\dim S_6^2(\Delta,p)
=L_6^2(\Delta,p)
=28+8I+10(b-3)+N_3+3N_2\ ?
\end{equation}
\end{question}
A counterexample must satisfy $\dim S_6^2>L_6^2$ with the actual local corrections. Generic attainment alone does not exclude such a special realization.

\begin{question}[Sextic jumps with fixed local data?]\label{q:j6}
Does there exist a fixed abstract triangulated disk $\Delta$ with two legal realizations $p,p'$ having equal local direction types at corresponding interior vertices but
\[
\dim S_6^2(\Delta,p)\ne\dim S_6^2(\Delta,p')\ ?
\]
In the phase notation, is $(6,2)\in\JJ$?
\end{question}

Universal attainment and geometric jumps with fixed local data are mutually exclusive. To see this, fix $d\ge3$ and two realizations $p,p'$ of the same triangulation with equal vertexwise local direction types. Their combinatorial counts and $N_2,N_3$ agree, so their actual lower bounds agree. If $(d,2)\in\UU$, then
\[
\dim S_d^2(\Delta,p)=L_d^2(\Delta,p)
=L_d^2(\Delta,p')=\dim S_d^2(\Delta,p').
\]
Their dimensions therefore cannot jump, proving $\UU\cap\JJ=\varnothing$. In particular, a positive answer to Question~\ref{q:s6} forces a negative answer to Question~\ref{q:j6}, while a sextic jump would disprove universal sextic attainment.

A negative answer to Question~\ref{q:s6} gives a special realization that fails to attain the bound. To establish a sextic jump, one must also find a realization of the same triangulation with the same vertexwise local direction types and a different dimension. Generic attainment alone does not ensure that the prescribed set of local types contains such a comparison realization.

The jumps in degrees three, four, and five exclude universal attainment in those degrees. Theorem~\ref{thm:universal} and Hong's theorem for $d\ge8$ \citep{Hong91} give universal attainment in every degree $d\ge7$. These results leave exactly the following two alternatives.
\begin{corollary}[The universal frontier: two alternatives]\label{thm:ufrontier}
In the degree range $d\ge3$ considered here, there are exactly two possibilities:
\[
\UU=\{(d,2):d\ge6\}\quad\text{or}\quad
\UU=\{(d,2):d\ge7\}.
\]
Thus $u(2)\in\{6,7\}$, where $u(2)$ is the least $d_0\ge3$ such that $(d,2)\in\UU$ for every $d\ge d_0$.
\end{corollary}
If the sextic formula is universal, the first alternative holds. Any sextic counterexample establishes the second alternative, whether or not it also leads to a fixed-local-data jump. The two question marks in Figure~\ref{fig:phase} record Questions~\ref{q:s6} and~\ref{q:j6} in the universal and geometric-jump rows, respectively.
\FloatBarrier

\Needspace{32\baselineskip}
\section{Conclusions}\label{sec:conclusion}
We prove that Schumaker's lower bound, including its local geometric corrections, is attained in degree seven for every legal realization of every triangulated disk. We also prove generic attainment in degrees five and six for every fixed combinatorial triangulation. The examples based on the known chord and boundary-pair mechanisms give explicit polynomial conditions certifying generic failure in degrees three and four. Together with Hong's theorem for degrees at least eight \citep{Hong91}, these results give universal attainment in every degree at least seven and an exact generic threshold of five.

We construct geometric jumps in degrees three and five, with different dimensions but the same local direction types at corresponding interior vertices. We also present the classical geometric jump in degree four. The realization pairs in degrees four and five preserve all orientation signs. Universal attainment in degree seven rules out geometric jumps with fixed local direction types at that degree.

In degree six, universal attainment and the existence of such geometric jumps remain open questions. Future work will focus on geometric conditions for strict excess and on dimension variation with fixed local direction types.
\FloatBarrier

\begingroup
\makeatletter
\renewcommand\section{\@startsection{section}{1}{\z@}%
  {-1.5ex\@plus-.5ex\@minus-.2ex}%
  {.8ex\@plus.2ex}%
  {\normalfont\Large\bfseries}}
\makeatother
\section*{Acknowledgements}
\begingroup\small\emergencystretch=1em
This work was supported by the National Key R\&D Program of China (Nos. 2022YFA1005200 and 2022YFA1005201), the NSF of China (No. 12171453), and the Major Project of Science and Technology Innovation Tackling Plan of Anhui Province (No. 202423e09050003).
\par\endgroup

\section*{Declarations}
The authors declare that they have no conflict of interest.
\par\endgroup

\clearpage
\appendix
\renewcommand{\thetheorem}{\Alph{section}.\arabic{theorem}}
\renewcommand{\theequation}{\Alph{section}.\arabic{equation}}
\renewcommand{\thesubsection}{\Alph{section}.\arabic{subsection}}
\renewcommand{\thesubsubsection}{\Alph{section}.\arabic{subsection}.\arabic{subsubsection}}
\section{Complete proof of the universal septic formula}\label{app:u7}\label{sec:app-u7}

\subsection{Proof outline: exclude a nonzero row relation on the actual target}\label{u7:sec:intro}

Fix any legal real-coordinate realization. We prove surjectivity of the septic cofactor map onto its actual local-ideal target, defined in Section~\ref{u7:sec:homology}. Choose a coefficient matrix $M_7$ and prove that its rows are independent, as follows.
\begin{enumerate}
\item \textbf{A row relation gives a nonzero dual state.} If $a\ne0$ satisfies $a^{\mathsf T}M_7=0$, it defines a nonzero linear functional on the actual target $\bigoplus_v(J_v)_7$ annihilating the image of $\Phi_7$. Representing it by differential pairing gives a compatible state $(H_v)$, still nonzero after quotienting by the complete local annihilators.
\item \textbf{The state gives finite nonzero scalar support.} Common traces on each actual face determine a unique face potential $K_\tau$. The leading term of $H_v-K_\tau$, normalized by area, gives a vertex scalar $\mu_v$. At two-direction vertices, normalizing $\mu_v$ to zero only changes the representative of the same effective class. If all normalized scalars vanish, the original functional vanishes. A nonzero row relation therefore gives nonempty finite scalar support.
\item \textbf{A farthest point gives a complete convex-hull star.} Choose a center outside the support and a farthest support point. The tangent to the corresponding circle puts all other nonzero-scalar points strictly on one side. Area transfer forces exactly two consecutive zero-scalar neighbors in its complete neighbor cycle; the remaining nonzero neighbors lie in a consecutive sector of angle less than $\pi$.
\item \textbf{The closed cycle excludes the star.} Extract cyclic cubic polynomials from the original state, keeping independent labels for the two exterior nodes. The Hermite equations, quadratic companion identity, and full period force positive endpoint slopes. Recurrence makes all inner slopes positive, contradicting the closed-cycle flux identity. Thus no nonzero row relation exists, $M_7$ has full row rank, and $\Phi_7$ is surjective.
\end{enumerate}

The matrix $M_7$ uses bases of the actual local-ideal target.

\subsection{Nonzero row relations and effective dual states}\label{u7:sec:homology}

We retain the notation of Section~\ref{sec:prelim}. Let $R=\RR[x,y,z]$, with $z$ the homogenizing variable. Subscripts on graded polynomial rings and homogeneous ideals denote the homogeneous component of the indicated degree, so $R_d$ consists of forms of degree exactly $d$. Choose a homogeneous line equation $\ell_e=0$ for each edge and define the edge and vertex ideals by
\begin{equation}\label{u7:eq:ideals}
J(e)=(\ell_e^3),\qquad J_v=\sum_{e\ni v}J(e).
\end{equation}
For the cofactor map
\begin{equation}\label{u7:eq:ideal-map}
\Phi_7:\mathcal E_7:=\bigoplus_{e\in\ee}R_4
\longrightarrow\mathcal V_7:=\bigoplus_{v\in\ve}(J_v)_7
\end{equation}
choose source monomials and a basis of each actual local-ideal target $(J_v)_7$, combined into a direct-sum basis $\mathscr B$. Its coefficient matrix is
\begin{equation}\label{u7:eq:coefficient-matrix}
[\Phi_7q]_{\mathscr B}=M_7[q],\qquad
M_7\in\RR^{m\times15E},\qquad
m=\sum_{v\in\ve}\dim (J_v)_7.
\end{equation}
Here $[q]$ is the concatenated source monomial coefficient vector, and $[\Phi_7q]_{\mathscr B}$ is the target coordinate vector in $\mathscr B$. Also, $(\Phi_7q)_v=\sum_{e\ni v}\epsilon_{ve}\ell_e^3q_e$, with coherent vertex-circuit incidence signs. Each interior edge has one shared cofactor; boundary vertices have no target component. The classical cofactor count gives \citep[Sections 2--3]{MourrainVillamizar13}
\begin{equation}\label{u7:eq:def-epsilon}
\eps_7:=\dim S_7^2-L_7^2
=\dim\coker\Phi_7=m-\rank M_7.
\end{equation}
The objective is therefore full row rank of $M_7$.

Suppose instead that $a\ne0$ satisfies $a^{\mathsf T}M_7=0$. It induces on the actual target
\begin{equation}\label{u7:eq:row-functional}
\Lambda_a(y)=a^{\mathsf T}[y]_{\mathscr B},
\qquad \Lambda_a\ne0,\qquad \Lambda_a\circ\Phi_7=0.
\end{equation}
Take $\calT=\RR[X,Y,Z]$ and the nondegenerate differential pairing
\begin{equation}\label{u7:eq:pairing}
\langle p,H\rangle
=p(\partial_X,\partial_Y,\partial_Z)H\big|_{(0,0,0)}
\qquad(p\in R_d,\ H\in\calT_d).
\end{equation}
Decompose $\Lambda_a$ by vertex and extend each functional on $(J_v)_7$ to $R_7$, representing it by $H_v\in\calT_7$. The freedom in this choice is precisely the full local annihilator
\begin{equation}\label{u7:eq:annihilator}
\calA_v=(J_v)_7^\perp
=\bigcap_{e\ni v}\ker(D_e^3:\calT_7\to\calT_4),
\end{equation}
For $\ell_e=\alpha x+\beta y+\gamma z$, write
$D_e=\alpha\partial_X+\beta\partial_Y+\gamma\partial_Z$.

Retain an arbitrary cofactor $q\in R_4$ on only one source edge. The identity
$\langle\ell_e^3q,H\rangle=\langle q,D_e^3H\rangle$
together with \eqref{u7:eq:row-functional} yields
\begin{equation}\label{u7:eq:trace-match}
D_e^3H_u=D_e^3H_v.
\end{equation}
If an edge has only one interior endpoint, its third-order trace there vanishes. Set boundary states $H_b=0$ as an auxiliary convention, so \eqref{u7:eq:trace-match} applies to every actual edge. This imposes no boundary condition on the original splines.

Let $\calZ_\Delta$ be the space of states satisfying these conditions. Quotienting by the local annihilators gives
\begin{equation}\label{u7:eq:dual-defect}
\calD_\Delta
=\calZ_\Delta\big/\bigoplus_{v\in\ve}\calA_v
\cong\ker(\Phi_7^\vee)
\cong(\coker\Phi_7)^\vee.
\end{equation}
Since $\Lambda_a\ne0$, the effective class is nonzero; changing the annihilator component of $H_v$ does not change the functional. Conversely, the effective class is zero exactly when each representative lies in its full local annihilator, equivalently when all actual edge traces vanish. We derive a geometric contradiction from this nonzero class.

\subsubsection{Local direction normal forms for normalization}
Associate with $v=(x_v,y_v)$ the fixed point linear form
\begin{equation}\label{u7:eq:point-form}
V=x_vX+y_vY+Z.
\end{equation}
Every edge operator through $v$ annihilates $V$. Two distinct directions can be normalized by an invertible linear change to $\partial_S,\partial_T$, with $(S,T,V)$ independent coordinates. The following calculation gives the actual target dimensions and explains why two-direction normalization preserves the effective class.

\begin{lemma}[Local dimensions and direction normal forms]\label{u7:lem:local}
Every interior vertex has at least two supporting directions, and
\begin{equation}\label{u7:eq:local-ranks}
\dim (J_v)_7=\begin{cases}
27,&g_v=2,\\
29,&g_v=3,\\
30,&g_v\ge4.
\end{cases}
\end{equation}
If $g_v\ge3$, then $\calA_v\subseteq V^4\calT_3$. If $g_v=2$, the annihilator contains $V^3S^2T^2$, which does not belong to $V^4\calT_3$.
\end{lemma}
\begin{proof}
A complete interior star covers a neighborhood of its center and hence cannot have only one supporting direction. For two directions, the common annihilator has the following nine basis elements:
\[
S^iT^jV^{7-i-j},\qquad0\le i,j\le2.
\]
Modulo $V^5\calT_2$, the remaining three monomials are
\[
S^2TV^4,\qquad ST^2V^4,\qquad S^2T^2V^3.
\]
Normalize a third distinct direction to $\partial_S+\partial_T$. For
\[
H=(aS^2T+bST^2)V^4+cS^2T^2V^3,
\]
direct differentiation gives
\[
(\partial_S+\partial_T)^3H
=6(a+b)V^4+12c(S+T)V^3.
\]
Thus the full three-direction annihilator is
\begin{equation}\label{u7:eq:three-ann}
\calA_v=V^5\calT_2\oplus\RR V^4ST(S-T),
\end{equation}
which has dimension seven. A fourth distinct direction can be written as $\partial_S+\lambda\partial_T$, with $\lambda\ne0,1$, and
\[
(\partial_S+\lambda\partial_T)^3[V^4ST(S-T)]
=6\lambda(1-\lambda)V^4\ne0.
\]
Four or more directions therefore leave only the six-dimensional space $V^5\calT_2$. Nondegeneracy of the pairing gives $\dim (J_v)_7=36-\dim\calA_v$, proving the assertions.
\end{proof}

By Lemma~\ref{u7:lem:local}, the actual target has row count
\begin{equation}\label{u7:eq:target-dimension}
m=27N_2+29N_3+30(I-N_2-N_3)
=30I-3N_2-N_3.
\end{equation}
Combining \eqref{u7:eq:def-epsilon} and \eqref{u7:eq:dual-defect},
\begin{equation}\label{u7:eq:epsilon-coker}
\eps_7=\dim\coker\Phi_7
=\dim\calD_\Delta=m-\rank M_7\ge0.
\end{equation}
If $m=0$, the criterion already settles the assertion. Otherwise fix the nonzero effective state induced by a nonzero row relation in \eqref{u7:eq:dual-defect} and derive a contradiction.

\subsection{Unique face potentials and area-normalized vertex scalars}\label{u7:sec:scalar}

\subsubsection{The common polynomial on an actual triangular face}
For three actual points $u,v,w$, define the signed doubled area by
\begin{equation}\label{u7:eq:area}
[u,v,w]=\det\begin{pmatrix}x_u&y_u&1\\x_v&y_v&1\\x_w&y_w&1\end{pmatrix}
=\det(v-u,w-u).
\end{equation}
Throughout this section, $U,V,W$ denote the point linear forms in~\eqref{u7:eq:point-form}. Nondegeneracy of the triangle is equivalent to linear independence of $U,V,W$.

\begin{lemma}[Unique common face potential]\label{u7:lem:face}
For every state satisfying~\eqref{u7:eq:trace-match} and every actual triangular face $\tau$, there is a unique $K_\tau\in\calT_7$ whose traces on the three face edges equal the corresponding common traces. We call $K_\tau$ the common face potential.
\end{lemma}
\begin{proof}
The three edge linear forms of a nondegenerate triangle are independent, as are their first-order differential operators. After an invertible linear change of variables, write these operators as $\partial_1,\partial_2,\partial_3$, and the corresponding quartic traces as $G_1,G_2,G_3$. Every pair of face edges shares a vertex state, so
\[
\partial_j^3G_i=\partial_i^3G_j\qquad(i\ne j).
\]
Recover the coefficients of the septic monomials one at a time. For an exponent vector $(a_1,a_2,a_3)$ with $a_1+a_2+a_3=7$, at least one $a_i$ is at least three. The corresponding coefficient in $G_i$, divided by $a_i(a_i-1)(a_i-2)$, uniquely determines the septic coefficient. If two exponents are at least three, cross-compatibility makes the two determinations agree. Three exponents cannot all be at least three, since that would require total degree at least nine. Every coefficient is therefore determined consistently, proving both existence and uniqueness.
\end{proof}

At vertex $v$ of $\tau=(u,v,w)$, the two incident face-edge operators cubed annihilate $H_v-K_\tau$. In point-form coordinates $(U,W,V)$, these operators are nonzero multiples of $\partial_U,\partial_W$. The two-direction annihilator form therefore gives a unique real $\lambda_{v,\tau}$ with
\begin{equation}\label{u7:eq:lambda}
H_v-K_\tau\equiv
\lambda_{v,\tau}V^3U^2W^2\pmod{V^4\calT_3}.
\end{equation}
This leading term uses $7=3+2+2$. At a boundary vertex, the same representation is used with $H_v=0$.

\subsubsection{Across-face transformations and area transfer}
Consider two actual faces sharing the edge $uv$:
\[
\tau=(u,v,w),\qquad \tau'=(u,v,w').
\]
Since $U,V,W$ are independent, there is a unique expression
\begin{equation}\label{u7:eq:neighbor-coordinate}
W'=sW+pU+qV,\qquad s\ne0.
\end{equation}
The condition $s\ne0$ follows from nondegeneracy of the second face. Neither $p,q\ne0$ nor $s>0$ is assumed.

\begin{figure}[tbp]
\centering
\begin{tikzpicture}[scale=1.02,line cap=round,line join=round]
\coordinate (u) at (-1.7,0);\coordinate (v) at (1.7,0);
\coordinate (w) at (-.35,1.5);\coordinate (z) at (.5,-1.35);
\draw (u)--(w)--(v)--(z)--cycle;\draw[thick] (u)--(v);
\foreach \p in {u,v,w,z}\fill (\p) circle (1.4pt);
\node[left] at (u) {$u$};\node[right] at (v) {$v$};
\node[above] at (w) {$w$};\node[below] at (z) {$w'$};
\node at (0,.52) {$\tau$};\node at (.15,-.50) {$\tau'$};
\node[right,align=left,text width=6.9cm] at (2.2,.1)
 {$W'=sW+pU+qV,\quad s\ne0$.\\[5pt]
 Both faces are nondegenerate.\\
 $p,q$ and the vertex scalars\\
 may all vanish.};
\end{tikzpicture}
\caption{Two actual triangles sharing an interior edge. Leading coefficients transform by $s^{-2}$. Normalizing by the square of the signed doubled area gives one scalar per vertex.}
\label{u7:fig:adjacent}
\end{figure}
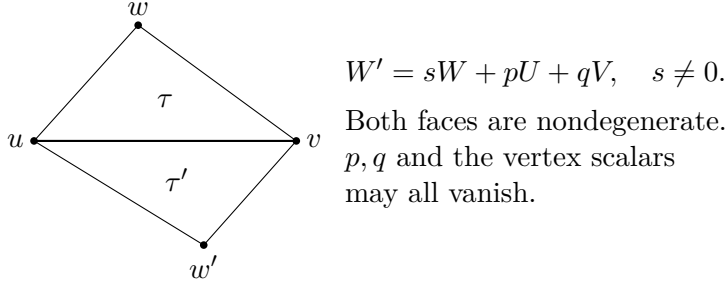

\begin{lemma}[Transformation of leading coefficients]\label{u7:lem:transition}
We have
\begin{equation}\label{u7:eq:lambda-transition}
\lambda_{v,\tau'}=\frac{\lambda_{v,\tau}}{s^2},
\qquad
\lambda_{u,\tau'}=\frac{\lambda_{u,\tau}}{s^2},
\qquad
\lambda_{v,\tau}p=\lambda_{u,\tau}q.
\end{equation}
Moreover, if $J=K_{\tau'}-K_\tau$, then
\begin{equation}\label{u7:eq:full-face-jump}
J=U^3V^3\left(
-\frac{2\lambda_{v,\tau}p}{s}W
-\frac{\lambda_{v,\tau}p^2}{s^2}U
-\frac{\lambda_{u,\tau}q^2}{s^2}V
\right).
\end{equation}
\end{lemma}
\begin{proof}
At the shared vertex $u$, each face potential differs from $H_u$ by a polynomial divisible by $U^3$; the analogous statement holds at $v$. Thus $J$ is divisible by both $U^3$ and $V^3$. The points are distinct, so $U,V$ are relatively prime. Since $J$ has degree seven,
\[
J=U^3V^3(\alpha W+\beta U+\gamma V).
\]
Subtract the two versions of~\eqref{u7:eq:lambda} at $v$, divide by $V^3U^2$, and reduce modulo $V$ to obtain
\[
U(\alpha W+\beta U)
=\lambda_{v,\tau}W^2-\lambda_{v,\tau'}(sW+pU)^2.
\]
Comparison of the coefficients of $W^2,UW,U^2$ gives
\[
\lambda_{v,\tau'}=\lambda_{v,\tau}/s^2,
\quad\alpha=-2\lambda_{v,\tau}p/s,
\quad\beta=-\lambda_{v,\tau}p^2/s^2.
\]
Comparison at $u$ gives the second transformation law and
\[
\alpha=-2\lambda_{u,\tau}q/s,
\qquad\gamma=-\lambda_{u,\tau}q^2/s^2.
\]
Equality of the two expressions for $\alpha$ yields the last condition in~\eqref{u7:eq:lambda-transition}; the remaining coefficients give~\eqref{u7:eq:full-face-jump}. No division by $p,q$, or by a leading coefficient, has been used.
\end{proof}

Equation~\eqref{u7:eq:neighbor-coordinate} gives $[u,v,w']=s[u,v,w]$. Define
\begin{equation}\label{u7:eq:mu}
\mu_v=[u,v,w]^2\lambda_{v,\tau}.
\end{equation}
Lemma~\ref{u7:lem:transition} shows that this scalar is single-valued on the complete face star of $v$. At a boundary vertex the incident faces form a connected chain, and at an interior vertex they form a closed cycle. Thus $\mu_v$ is independent of the chosen incident face.

Also,
\[
[u,w,w']=-q[u,v,w],\qquad
[v,w,w']=p[u,v,w],
\]
so Lemma~\ref{u7:lem:transition} gives the area-transfer identity on every interior edge:
\begin{equation}\label{u7:eq:area-transfer}
\boxed{\mu_u[u,w,w']+\mu_v[v,w,w']=0.}
\end{equation}
This is used only as a necessary condition; we do not assume every scalar field satisfying it lifts to an actual state.

\begin{lemma}[Scaling point linear forms]\label{u7:lem:scaling}
Replace the point forms by $c_uU,c_vV,c_wW$, with all three constants nonzero, and normalize using the determinant of the coefficient vectors of the rescaled forms. Then
\begin{equation}\label{u7:eq:mu-scaling}
\lambda'_{v,\tau}=\frac{\lambda_{v,\tau}}{c_v^3c_u^2c_w^2},
\qquad \mu'_v=\frac{\mu_v}{c_v}.
\end{equation}
In particular, scaling only neighboring point forms while fixing the central form leaves the central scalar unchanged.
\end{lemma}
\begin{proof}
The leading monomial in~\eqref{u7:eq:lambda} is multiplied by $c_v^3c_u^2c_w^2$, whereas the determinant of the point-form coefficient vectors is multiplied by $c_uc_vc_w$. Substituting its square into~\eqref{u7:eq:mu} gives the result. The constants may be negative; after rescaling, the determinant is computed from the rescaled coefficient vectors.
\end{proof}

\subsubsection{Boundary normalization and detection of nonzero classes}
Both endpoints of a boundary edge have zero state. Each of the other two edges of its incident face also has a boundary endpoint, so all three common traces vanish. Lemma~\ref{u7:lem:face} implies that the face potential vanishes, and hence so do the leading scalars at the boundary endpoints. Applying Lemma~\ref{u7:lem:transition} along each connected boundary-vertex face star gives
\begin{equation}\label{u7:eq:boundary-mu}
\mu_b=0\qquad(b\in V\setminus V^\circ).
\end{equation}

If $g_v\ge3$, Lemma~\ref{u7:lem:local} gives $\calA_v\subseteq V^4\calT_3$, so $\mu_v$ is independent of the representative of the effective class. If $g_v=2$, a properly embedded complete interior star consists of the four rays of two lines: a single ray cannot contain two distinct neighbors, and a complete neighborhood requires both directions of each line. The local polynomial $V^3S^2T^2$ is invisible to all four edge traces but has a nonzero leading coefficient on every incident face. A unique suitable multiple of it can therefore be added to $H_v$ to impose
\begin{equation}\label{u7:eq:two-gauge}
\mu_v=0\qquad(g_v=2).
\end{equation}
This change preserves all common edge traces. Uniqueness of face potentials means it also preserves every $K_\tau$ and all other vertex scalars. Normalizing a two-direction interior scalar to zero therefore only selects a representative of the same effective class, without adding a geometric assumption. We use this normalization throughout. Every nonzero-scalar point subsequently selected has $g_v\ne2$.

\begin{proposition}[Normalized scalars detect the effective class]\label{u7:prop:scalar-detect}
If a normalized state satisfies $\mu_v=0$ at every vertex, then its class in $\calD_\Delta$ is zero. Consequently, every nonzero effective class has nonempty normalized scalar support.
\end{proposition}
\begin{proof}
The condition $\mu_v=0$ is equivalent to $H_v-K_\tau\in V^4\calT_3$ in~\eqref{u7:eq:lambda}. For two adjacent faces with shared endpoints $u,v$, their potential difference therefore belongs to both $U^4\calT_3$ and $V^4\calT_3$. The nonproportional linear forms $U,V$ are relatively prime. A degree-seven polynomial divisible by $U^4V^4$ must be zero. Thus all adjacent face potentials agree.

The face-adjacency graph of a disk triangulation is connected, and a face incident to a boundary edge has zero potential. Hence every face potential vanishes. Every actual edge belongs to a face, so its common trace vanishes as well. Each interior state belongs to $\calA_v$, and the effective class is zero.
\end{proof}

\begin{remark}\label{u7:rem:scalar-support}
The scalar support $\{v:\mu_v\ne0\}$ and the support of nonzero edge traces are different objects. The equation $\mu_v=0$ does not, at an individual vertex, imply that all traces of $H_v$ vanish. Proposition~\ref{u7:prop:scalar-detect} uses the vanishing of the entire normalized scalar field. All subsequent convex hulls are taken from the scalar support.
\end{remark}

\subsection{The convex-hull star of nonzero scalar support}\label{u7:sec:hull}

Suppose that a nonzero effective class exists, and choose a normalized state representing it. Proposition~\ref{u7:prop:scalar-detect} shows that the finite set
\[
\mathcal S=\{v\in\Delta_0:\mu_v\ne0\}
\]
is nonempty. By \eqref{u7:eq:boundary-mu} and normalization \eqref{u7:eq:two-gauge}, every point of $\mathcal S$ is an interior vertex with $g_v\ge3$; two-direction vertices are already excluded.

Choose $o\notin\mathcal S$ and a farthest point $v\in\mathcal S$ from $o$. The circle centered at $o$ through $v$ bounds a closed disk containing $\mathcal S$; its tangent at $v$ therefore places all other support points strictly on one side, making $v$ a \emph{strictly exposed vertex}. Zero-scalar mesh vertices are still retained in the complete neighbor star.

It suffices to analyze this outermost point $v$. Translate it to the origin and list its neighbors counterclockwise as $r_0,\ldots,r_{k-1}$. Every consecutive ray pair bounds an actual face, so their angles have lifts satisfying
\begin{equation}\label{u7:eq:lifted-angle}
\theta_{i+k}=\theta_i+2\pi,
\qquad0<\theta_{i+1}-\theta_i<\pi.
\end{equation}
Two neighbors cannot lie on the same ray: the nearer would lie in the relative interior of the longer edge.

\begin{lemma}[Opposite rays adjacent to a zero-scalar neighbor]\label{u7:lem:zero-neighbor}
If $\mu_{r_i}=0$, then
\begin{equation}\label{u7:eq:opposite-rays}
\theta_{i+1}-\theta_{i-1}=\pi.
\end{equation}
\end{lemma}
\begin{proof}
The two actual opposite vertices of the edge $vr_i$ are $r_{i-1},r_{i+1}$. The area-transfer equation~\eqref{u7:eq:area-transfer}, together with $\mu_v\ne0$, gives
\[
[v,r_{i-1},r_{i+1}]=0.
\]
Thus these two neighbors are collinear with $v$. They cannot lie on the same ray: the two consecutive faces have positive oriented area, and the lifted angular difference from $r_{i-1}$ to $r_{i+1}$ lies strictly between zero and $2\pi$. The rays are therefore opposite, with angular difference exactly $\pi$.
\end{proof}

\begin{proposition}[Normal form of the complete convex-hull star]\label{u7:prop:hull-form}
After an orientation-preserving affine transformation, the complete neighbor cycle at $v$ has the form
\begin{equation}\label{u7:eq:hull-star}
(1,0)=p_0,\ p_1,\ldots,p_{n-1},\ p_n=(0,1),\ (-b,0),\ (0,-a),
\end{equation}
where $a,b>0$ and $n\ge2$. The first $n+1$ neighbors have nonzero scalars and strictly increasing directions. The intermediate points $p_1,\ldots,p_{n-1}$ lie in the strict first quadrant. The last two neighbors have zero scalars.
\end{proposition}
\begin{proof}
The neighbors cannot all have nonzero scalars: rays lying in a strict half-plane cannot form the complete star~\eqref{u7:eq:lifted-angle}. Nor can all neighbors have zero scalars. Repeated application of~\eqref{u7:eq:opposite-rays} then gives
\[
\theta_{i+3}=\theta_{i-1}+2\pi.
\]
Strict increase and periodicity force the cycle to have length four, with two pairs of opposite rays. This would give $g_v=2$, contrary to the normalization at a vertex with $\mu_v\ne0$.

Call a maximal consecutive sequence of zero-scalar neighbors a zero block. An isolated zero neighbor is impossible: its two nonzero neighbors would be opposite by~\eqref{u7:eq:opposite-rays}, and hence could not lie in the strict exposing half-plane. Three consecutive zero neighbors would likewise force a turn of exactly $2\pi$ in four steps, reducing to the excluded four-ray, two-direction case. The lifted-angle formulation also covers a block that crosses the chosen starting index.

Every zero block therefore has exactly two vertices. If their indices are $i,i+1$, two applications of~\eqref{u7:eq:opposite-rays} give
\begin{equation}\label{u7:eq:zero-block-angle}
\theta_{i+2}-\theta_{i-1}
=\pi+(\theta_i-\theta_{i-1})>\pi.
\end{equation}
For a zero block $r_i,r_{i+1}$ define its \emph{flanking angular interval}
\[
I_i=(\theta_{i-1},\theta_{i+2})
\]
using the lifted angles, with indices shifted by multiples of $k$ when necessary. Equation~\eqref{u7:eq:zero-block-angle} says that $|I_i|>\pi$. Distinct maximal zero blocks are separated by at least one nonzero neighbor. Their flanking intervals are therefore disjoint as open subintervals of one lifted turn: if exactly one nonzero neighbor separates the blocks, it is a common endpoint and is excluded; if more than one separates them, there is a positive gap. Two disjoint open intervals in a circle of total angular length $2\pi$ cannot both have length $>\pi$. Hence there cannot be two zero blocks. There is exactly one pair of consecutive zero neighbors, and the remaining nonzero neighbors form a consecutive block spanning less than a half-turn.

If this nonzero block had only one vertex, both zero neighbors would be forced onto the same opposite ray, violating uniqueness of a neighbor on each ray. If it had two vertices, the whole cycle would have four rays, with the zero neighbors opposite the two nonzero neighbors; again $g_v=2$. Thus the nonzero block has at least three vertices.

Use its first and last actual neighbor vectors as an oriented linear basis, sending them to $(1,0)$ and $(0,1)$. All intermediate rays lie in their positive cone, so their points lie in the strict first quadrant. The opposite-ray conditions at the two zero neighbors then put the last two points at $(-b,0),(0,-a)$, with $a,b>0$. This is~\eqref{u7:eq:hull-star}.
\end{proof}

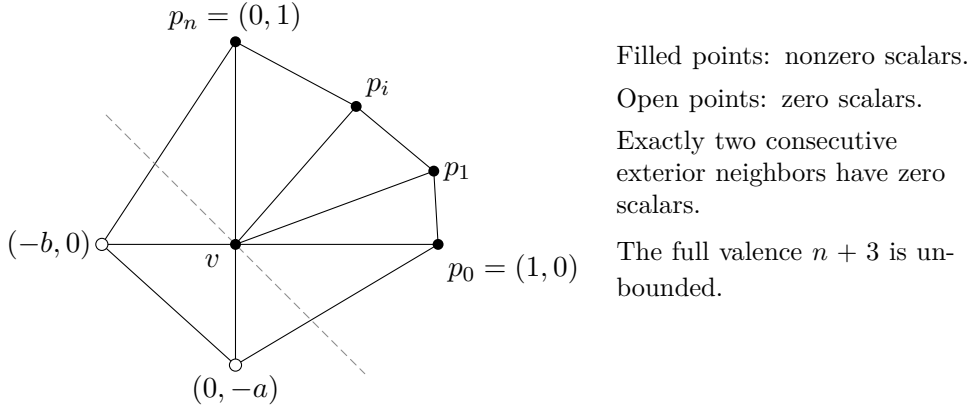
\begin{figure}[tbp]
\centering
\begin{tikzpicture}[scale=1.14,line join=round,line cap=round]
\coordinate (v) at (0,0);
\coordinate (p0) at (2.35,0);\coordinate (p1) at (2.3,.85);
\coordinate (p2) at (1.4,1.6);\coordinate (pn) at (0,2.35);
\coordinate (b) at (-1.55,0);\coordinate (a) at (0,-1.4);
\draw (p0)--(p1)--(p2)--(pn)--(b)--(a)--cycle;
\foreach \p in {p0,p1,p2,pn,b,a}\draw (v)--(\p);
\draw[densely dashed,black!45] (-1.5,1.5)--(1.5,-1.5);
\foreach \p in {v,p0,p1,p2,pn}\fill (\p) circle (1.8pt);
\foreach \p in {a,b}\draw[fill=white] (\p) circle (2pt);
\node[below left,xshift=-2pt] at (v) {$v$};
\node[below right] at (p0) {$p_0=(1,0)$};
\node[right] at (p1) {$p_1$};\node[above right] at (p2) {$p_i$};
\node[above] at (pn) {$p_n=(0,1)$};
\node[left] at (b) {$(-b,0)$};\node[below] at (a) {$(0,-a)$};
\node[align=left,text width=5.2cm,right] at (4.3,.85)
 {\small Filled points: nonzero scalars.\\[4pt]
 Open points: zero scalars.\\[4pt]
 Exactly two consecutive\\
 exterior neighbors have zero scalars.\\[4pt]
 The full valence $n+3$ is unbounded.};
\end{tikzpicture}
\caption{The complete star at a strictly exposed vertex. The dashed line strictly exposes the support; the sector from $p_0$ to $p_n$ contains all nonzero-scalar neighbors. Equal angles or neighbor distances are not assumed.}
\label{u7:fig:hull}
\end{figure}

Proposition~\ref{u7:prop:hull-form} does not assume that the original domain is convex: the convex hull is used only for the finite set of nonzero-scalar points. Nor is noncollinearity of that support assumed in advance. A singleton or a collinear support still has a strictly exposed vertex, whereas the argument requires at least three noncollinear nonzero neighbors and therefore excludes such supports. It remains to show that~\eqref{u7:eq:hull-star} cannot occur in an actual nonzero state.

\subsection{From a complete star to labeled cyclic cubic data}\label{u7:sec:cyclic}

Fix the compatible state and the star obtained in Proposition~\ref{u7:prop:hull-form}. We now extract cyclic cubic data from their face potentials.

\subsubsection{Signed parameters and normalized point forms}
Write the first group of neighbors as
\begin{equation}\label{u7:eq:radial-points}
p_i=\rho_i(1-t_i,t_i),\qquad
0=t_0<t_1<\cdots<t_n=1,\qquad
\rho_i>0,\quad\rho_0=\rho_n=1.
\end{equation}
For $i=0,\ldots,n-1$, put $h_i=t_{i+1}-t_i$. Then
\begin{equation}\label{u7:eq:positive-lengths}
h_i>0,\qquad \sum_{i=0}^{n-1}h_i=1,
\qquad 0<h_i<1.
\end{equation}
The last inequality uses $n\ge2$.

In the current affine coordinates, let $V$ be the central point form and $\xi,\eta$ the other two dual variables. Make the linear change
\[
X=\xi-\frac Vb,\qquad Y=\eta-\frac Va.
\]
Its determinant is one. The two exterior neighbors have point forms $-bX,-aY$; dividing by their nonzero scalar factors normalizes them to $X,Y$. Dividing each of the first group of point forms by $\rho_i$ gives
\begin{equation}\label{u7:eq:Li}
L_i=(1-t_i)X+t_iY+\zeta_iV,
\qquad
\zeta_i=\frac1{\rho_i}+\frac{1-t_i}{b}+\frac{t_i}{a}.
\end{equation}
In particular,
\begin{equation}\label{u7:eq:positive-endpoints}
\zeta_0=1+\frac1b>0,
\qquad \zeta_n=1+\frac1a>0.
\end{equation}
By Lemma~\ref{u7:lem:scaling}, these normalizations of the neighboring point forms, including the negative rescalings, leave the central scalar unchanged. Multiplying the entire state by a nonzero constant, we may normalize that scalar to one.

Set $N=n+3$, $p_{n+1}=(-b,0)$, and $p_{n+2}=(0,-a)$, and interpret indices cyclically modulo $N$ from now on, so $p_N=p_0$. Assign the two exterior points the \emph{labeled} parameters
\begin{equation}\label{u7:eq:outer-labels}
(t_{n+1},\zeta_{n+1})=(0,0),\qquad
(t_{n+2},\zeta_{n+2})=(1,0).
\end{equation}
Set $L_{n+1}=X$, $L_{n+2}=Y$, and $L_N=L_0$. Thus the first formula in~\eqref{u7:eq:Li} holds at every cyclic index. Set $t_N=t_0$ and define $h_i=t_{i+1}-t_i$ at every index. Thus
\begin{equation}\label{u7:eq:outer-lengths}
h_n=-1,\qquad h_{n+1}=1,\qquad h_{n+2}=-1.
\end{equation}
The numerical parameter values $0,1$ recur at different labels, but these labels represent different actual ray nodes. They must not be merged, and their slopes must not be identified. A negative $h_i$ records only a signed parameter length; it does not alter nondegeneracy of the original face. In coordinates $(X,Y,V)$, the normalized point forms satisfy
\[
\det(V,L_i,L_{i+1})=h_i\ne0.
\]

\subsubsection{Telescoping decomposition of face-potential differences}
Let $K_i$ be the common potential of the $i$th actual face $(v,p_i,p_{i+1})$. Define
\begin{equation}\label{u7:eq:slopes-zeta}
k_i=\frac{\zeta_{i+1}-\zeta_i}{h_i},\qquad
\delta_i=k_i-k_{i-1},
\end{equation}
and
\begin{equation}\label{u7:eq:G-C}
G_i=\frac{V^3L_i^2L_{i+1}^2}{h_i^2},
\qquad
\calC_i=\delta_iL_i^2
\left(\frac{L_{i+1}}{h_i}-\frac{L_{i-1}}{h_{i-1}}\right).
\end{equation}
The central scalar is one, so~\eqref{u7:eq:lambda} implies that $H_v-K_i-G_i$ is divisible by $V^4$.

\begin{lemma}[Exact local difference identity]\label{u7:lem:period-block}
At every cyclic index,
\begin{equation}\label{u7:eq:period-block}
K_i-K_{i-1}=G_{i-1}-G_i+V^4\calC_i,
\end{equation}
and therefore
\begin{equation}\label{u7:eq:Omega}
\Omega(X,Y,V):=\sum_{i=0}^{N-1}\calC_i(X,Y,V)=0.
\end{equation}
\end{lemma}
\begin{proof}
Temporarily write $a=h_{i-1}$ and $b=h_i$, and set
\[
s=-b/a,\qquad p=b\delta_i,\qquad q=1+b/a.
\]
Comparing the three point-form coefficients gives
\begin{equation}\label{u7:eq:Li-transition}
L_{i+1}=sL_{i-1}+pV+qL_i.
\end{equation}
Apply Lemma~\ref{u7:lem:transition} at the shared edge $vp_i$, taking the two shared point forms to be $V,L_i$ and the opposite point forms to be $L_{i-1},L_{i+1}$. The transition coefficients are exactly the $s,p,q$ in~\eqref{u7:eq:Li-transition}. The central leading coefficient on the preceding face is $1/a^2$. If $\lambda_i$ denotes the leading coefficient at the neighbor $p_i$ on that face, the last relation in~\eqref{u7:eq:lambda-transition} (equivalently the area-transfer identity) gives the product identity
\[
\lambda_i p=\frac{q}{a^2}.
\]
Substituting this product, together with $s=-b/a$, into the full jump formula~\eqref{u7:eq:full-face-jump} gives directly
\begin{equation}\label{u7:eq:radial-jump}
K_i-K_{i-1}
=-\frac q{b^2}V^3L_i^3(sL_{i-1}+L_{i+1}).
\end{equation}
Thus no division by $p,q$, by $\lambda_i$, or by any neighboring scalar is used; the formula remains valid when any of those quantities vanishes.

Subtract $G_{i-1}-G_i$ from the right-hand side and factor out $V^3L_i^2/b^2$. The remaining quadratic is
\begin{align*}
&L_{i+1}^2-s^2L_{i-1}^2-qL_i(sL_{i-1}+L_{i+1})\\
&\qquad=(L_{i+1}-sL_{i-1}-qL_i)(L_{i+1}+sL_{i-1})\\
&\qquad=pV(L_{i+1}+sL_{i-1}).
\end{align*}
Using $p=b\delta_i$ and $s=-b/a$ identifies the difference as $V^4\calC_i$. This proves~\eqref{u7:eq:period-block}, including the cases in which $\delta_i,p$, or $q$ vanishes. On the complete closed star, the cyclic differences of both $K_i$ and $G_i$ sum to zero. Hence $V^4\Omega=0$, and~\eqref{u7:eq:Omega} follows.
\end{proof}

\subsubsection{Cyclic cubics from a single central state}
The leading-term relation gives polynomials $F_i\in\calT_3$ such that
\begin{equation}\label{u7:eq:Fi}
H_v-K_i-G_i=V^4F_i(X,Y,V).
\end{equation}
The edge operator through the center and the $i$th neighbor annihilates $V,L_i$, and the operator through the $(i+1)$st neighbor annihilates $V,L_{i+1}$. The third-order traces of $H_v-K_i$ on these two edges vanish. The same is true for $G_i$, since the factor on which differentiation can act has degree only two. Thus the third-order operators associated with the two radial edges annihilate $F_i$.

In $(X,Y,V)$ coordinates, the $i$th direction operator can be scaled to
\[
D_i=-t_i\partial_X+(1-t_i)\partial_Y.
\]
For a homogeneous cubic $F_i$, one has $D_i^3F_i=6F_i(-t_i,1-t_i,0)$. Define
\begin{equation}\label{u7:eq:fi}
f_i(t)=F_i(-t,1-t,0).
\end{equation}
Then
\begin{equation}\label{u7:eq:fi-endpoints}
f_i(t_i)=f_i(t_{i+1})=0.
\end{equation}
These cubics are indexed by the cyclic order of actual faces. Distinct signed intervals can overlap, so the polynomials need not define an ordinary single-valued piecewise function on the real axis.

Equations~\eqref{u7:eq:period-block} and~\eqref{u7:eq:Fi} also give the full polynomial difference
\begin{equation}\label{u7:eq:Fi-jump}
F_i-F_{i-1}=-\calC_i.
\end{equation}
Let $\calC_i^{[j]}(t)$ be the coefficient of $V^j$ in $\calC_i(-t,1-t,V)$. Expanding \eqref{u7:eq:G-C} gives
\begin{equation}\label{u7:eq:C0}
\calC_i^{[0]}(t)=\delta_i\left[
\left(\frac1{h_i}-\frac1{h_{i-1}}\right)(t_i-t)^3
+2(t_i-t)^2\right].
\end{equation}
Consequently,
\begin{equation}\label{u7:eq:cubic-jump}
f_i-f_{i-1}=-\calC_i^{[0]}.
\end{equation}
The right-hand side has a double zero at $t_i$, so the derivatives of the two polynomials at the common labeled node agree: $f'_{i-1}(t_i)=f'_i(t_i)$. Set
\begin{equation}\label{u7:eq:ri}
r_i=f_{i-1}'(t_i)=f_i'(t_i).
\end{equation}

\subsection{A quadratic companion and the closed-star flux contradiction}\label{u7:sec:flux}

We prove that the star~\eqref{u7:eq:hull-star} cannot carry a state whose central scalar is one. The argument applies uniformly to every $n\ge2$.

\subsubsection{The original Hermite equations and two exterior nodes}
On the signed interval from $t_-$ to $t_+=t_-+h$, the unique cubic with zero endpoint values and endpoint slopes $r_-,r_+$ is
\begin{equation}\label{u7:eq:hermite-poly}
f(t)=r_-u-\frac{2r_-+r_+}{h}u^2
+\frac{r_-+r_+}{h^2}u^3,
\qquad u=t-t_-.
\end{equation}
Only $h\ne0$ is required. The derivatives are
\begin{equation}\label{u7:eq:hermite-derivatives}
f''(t_-)=-\frac{4r_-+2r_+}{h},\quad
f''(t_+)=\frac{2r_-+4r_+}{h},\quad
f'''=\frac{6(r_-+r_+)}{h^2}.
\end{equation}

\begin{lemma}[Cyclic recurrence and flux]\label{u7:lem:hermite-recurrence}
The actual cyclic data satisfy
\begin{align}
\frac{r_{i-1}}{h_{i-1}}
+2\left(\frac1{h_{i-1}}+\frac1{h_i}\right)r_i
+\frac{r_{i+1}}{h_i}&=2\delta_i,
\label{u7:eq:hermite-A}\\
\frac{r_i+r_{i+1}}{h_i^2}
-\frac{r_{i-1}+r_i}{h_{i-1}^2}
&=\delta_i\left(\frac1{h_i}-\frac1{h_{i-1}}\right).
\label{u7:eq:hermite-B}
\end{align}
For $i=0,\ldots,n$,
\begin{equation}\label{u7:eq:alternating-rec}
\frac{r_{i+1}}{h_i}=\frac{r_{i-1}}{h_{i-1}}.
\end{equation}
while the two exterior nodes give
\begin{equation}\label{u7:eq:outer-r}
r_{n+1}=r_0-2\zeta_0,\qquad
r_{n+2}=r_n-2\zeta_n.
\end{equation}
Moreover,
\begin{equation}\label{u7:eq:flux}
\boxed{\sum_{i=0}^{n-1}\frac{r_i+r_{i+1}}{h_i}=r_0+r_n.}
\end{equation}
\end{lemma}
\begin{proof}
Use \eqref{u7:eq:hermite-derivatives} to compare second and third derivatives in \eqref{u7:eq:cubic-jump}, obtaining \eqref{u7:eq:hermite-A}--\eqref{u7:eq:hermite-B}. Eliminating $\delta_i$ leaves exactly
\begin{equation}\label{u7:eq:factor-recurrence}
\frac{h_{i-1}+h_i}{2h_{i-1}h_i}
\left(\frac{r_{i+1}}{h_i}-\frac{r_{i-1}}{h_{i-1}}\right)=0.
\end{equation}
For $1\le i\le n-1$, both lengths are positive. At $i=0$ their sum is $-1+h_0$, and at $i=n$ it is $h_{n-1}-1$; these are also nonzero by~\eqref{u7:eq:positive-lengths}. Thus~\eqref{u7:eq:alternating-rec} holds for $i=0,\ldots,n$.

At $i=n+1,n+2$, however, the adjacent signed lengths sum to zero, so this division is not allowed. Equations~\eqref{u7:eq:outer-labels}--\eqref{u7:eq:slopes-zeta} give
\[
\delta_{n+1}=-\zeta_n,\qquad
\delta_{n+2}=-\zeta_0.
\]
Substituting the exterior lengths $(-1,1,-1)$ directly into~\eqref{u7:eq:hermite-A} yields the two identities in~\eqref{u7:eq:outer-r}.

Finally, sum~\eqref{u7:eq:hermite-A} over all cyclic indices. The right-hand side vanishes because $\sum_i\delta_i=0$. Each signed-interval term $(r_i+r_{i+1})/h_i$ occurs three times on the left, so
\[
\sum_{i=0}^{N-1}\frac{r_i+r_{i+1}}{h_i}=0.
\]
The three exterior terms are
\[
-(r_n+r_{n+1})+(r_{n+1}+r_{n+2})-(r_{n+2}+r_0)
=-r_n-r_0.
\]
Rearranging gives \eqref{u7:eq:flux}.
\end{proof}

\begin{remark}
Equation~\eqref{u7:eq:outer-r} does not identify slopes at the repeated parameter values $t_{n+1}=t_0$ or $t_{n+2}=t_n$: these nodes occupy different positions in the cycle. Such an identification would discard the two genuine endpoint constraints carried by the original complete star.
\end{remark}

\subsubsection{The quadratic companion and its jump}
Let $\zeta_i(t)=\zeta_i+k_i(t-t_i)$ be the affine radial interpolant on the $i$th signed interval. Define
\begin{equation}\label{u7:eq:quadratic-companion}
Q_i(t)=\zeta_i(t)f_i'(t)-3k_if_i(t)
-\bigl(f_i'(t)\bigr)^2+\frac32f_i(t)f_i''(t).
\end{equation}
We call this the \emph{quadratic companion polynomial} of the cubic segment. Indeed, if $f=At^3+Bt^2+Ct+D$, then
\begin{equation}\label{u7:eq:hessian-identity}
(f')^2-\frac32ff''
=(B^2-3AC)t^2+(BC-9AD)t+C^2-3BD.
\end{equation}
The cubic terms in $\zeta_i(t)f_i'-3k_if_i$ also cancel, so $\deg Q_i\le2$.

Extracting the coefficient linear in $V$ from~\eqref{u7:eq:G-C} gives
\begin{equation}\label{u7:eq:C1}
\begin{split}
\calC_i^{[1]}(t)=\delta_i\Bigg[&
\left\{3\zeta_i\left(\frac1{h_i}-\frac1{h_{i-1}}\right)
+k_i+k_{i-1}\right\}(t_i-t)^2\\
&+4\zeta_i(t_i-t)\Bigg].
\end{split}
\end{equation}

\begin{lemma}[Quadratic jump identity]\label{u7:lem:quadratic-jump}
Under the Hermite equations~\eqref{u7:eq:hermite-A}--\eqref{u7:eq:hermite-B},
\begin{equation}\label{u7:eq:quadratic-jump}
\boxed{
Q_i-Q_{i-1}-\calC_i^{[1]}
=\frac32\delta_i
\left(\frac{r_{i-1}}{h_{i-1}}-\frac{r_{i+1}}{h_i}\right)(t-t_i)^2.
}
\end{equation}
This identity also holds at both exterior nodes.
\end{lemma}
\begin{proof}
All three polynomials on the left have degree at most two. It therefore suffices to compare their values and first two derivatives at the node. Translate the node to $t_i=0$. Write $a,b$ for the left and right signed lengths, $r$ for the common slope, $r_-,r_+$ for the outer endpoint slopes, $z$ for the radial value at the node, and $l,k$ for the left and right radial slopes. Put $\delta=k-l$. Subscripts $-,+$ refer to the two cubic segments, and $[g]=g_+(0)-g_-(0)$.

Since $f_\pm(0)=0$ and $f_\pm'(0)=r$, one has $Q_+(0)=Q_-(0)=zr-r^2$, while $\calC^{[1]}(0)=0$. The original polynomial difference gives
\[
[f'']=-4\delta,\qquad
[f''']=6\delta(1/b-1/a).
\]
Differentiating~\eqref{u7:eq:quadratic-companion} and evaluating at the node gives
\[
Q'=(z-r/2)f''-2kr,
\qquad
Q''=(z+r)f'''-kf''-\tfrac12(f'')^2.
\]
Consequently $[Q']=-4z\delta=(\calC^{[1]})'(0)$, so the linear coefficients agree.

Set $M=f_-''(0)=(2r_-+4r)/a$, $P=f_+''(0)=-(4r+2r_+)/b$, and $d=1/b-1/a$. Since $P-M=-4\delta$,
\begin{align*}
[Q'']-(\calC^{[1]})''
&=6r\delta d-kP+lM-\tfrac12(P^2-M^2)-2\delta(k+l)\\
&=\delta\{6rd+P+2M-2\delta\}\\
&=\delta\{6rd+3M-6\delta\}.
\end{align*}
Substitution of~\eqref{u7:eq:hermite-A}, namely
\[
2\delta=r_-/a+2(1/a+1/b)r+r_+/b,
\]
reduces the last line to $3\delta(r_-/a-r_+/b)$. This is the second derivative of the right-hand side of~\eqref{u7:eq:quadratic-jump}. The values and first two derivatives agree, proving the identity. Only the nonzero lengths $a,b$ have been used as divisors; no division by $a+b$ occurs.
\end{proof}

\Needspace{10\baselineskip}
\subsubsection{Endpoint reduction of the full period and the main proof}
For a polynomial in $V$, the notation $[V]$ extracts its coefficient of $V$.
\begin{proposition}[Endpoint flux identity]\label{u7:prop:endpoint}
For the actual complete star above,
\begin{equation}\label{u7:eq:endpoint-polynomial}
\begin{split}
[V]\Omega(-t,1-t,V)
={}&3\zeta_0(r_0-\zeta_0)(t-1)^2\\
&-3\zeta_n(r_n-\zeta_n)t^2.
\end{split}
\end{equation}
Consequently, $\Omega=0$ forces
\begin{equation}\label{u7:eq:endpoint-positive}
r_0=\zeta_0>0,\qquad r_n=\zeta_n>0.
\end{equation}
\end{proposition}
\begin{proof}
Sum~\eqref{u7:eq:quadratic-jump} around the complete cycle. The differences $Q_i-Q_{i-1}$ telescope, leaving $-[V]\Omega(-t,1-t,V)$ on the left. For $i=0,\ldots,n$, the right-hand side vanishes by~\eqref{u7:eq:alternating-rec}. Only the two exterior nodes remain.

At $i=n+1$, one has $t_i=0$ and $\delta_i=-\zeta_n$, and
\[
\frac{r_{i-1}}{h_{i-1}}-\frac{r_{i+1}}{h_i}
=-r_n-r_{n+2}=2(\zeta_n-r_n).
\]
At $i=n+2$, one has $t_i=1$ and $\delta_i=-\zeta_0$, and
\[
\frac{r_{i-1}}{h_{i-1}}-\frac{r_{i+1}}{h_i}
=r_{n+1}+r_0=2(r_0-\zeta_0).
\]
Substitution into the summed identity and rearrangement give~\eqref{u7:eq:endpoint-polynomial}. The full period vanishes by~\eqref{u7:eq:Omega}. Evaluating at $t=0,1$ and using the strict positivity in~\eqref{u7:eq:positive-endpoints} yields~\eqref{u7:eq:endpoint-positive}.
\end{proof}

\begin{theorem}[Exclusion of the complete convex-hull star]\label{u7:thm:local-obstruction}
The complete star of Proposition~\ref{u7:prop:hull-form} cannot arise from a compatible state with nonzero normalized central scalar.
\end{theorem}
\begin{proof}
Suppose that it does, and normalize the central scalar to one. Proposition~\ref{u7:prop:endpoint} gives $r_0=\zeta_0>0$ and $r_n=\zeta_n>0$. Equation~\eqref{u7:eq:outer-r} gives $r_{n+2}=-\zeta_n$. Taking $i=0$ in~\eqref{u7:eq:alternating-rec} and using $h_{n+2}=-1$, we obtain
\[
\frac{r_1}{h_0}=\frac{r_{n+2}}{-1}=\zeta_n>0.
\]
Thus $r_0,r_1>0$. For $1\le i\le n-1$, the recurrence
\[
r_{i+1}=\frac{h_i}{h_{i-1}}r_{i-1}
\]
has a positive coefficient, so every $r_0,\ldots,r_n$ is strictly positive.

Since $0<h_i<1$,
\begin{align*}
\sum_{i=0}^{n-1}\frac{r_i+r_{i+1}}{h_i}
&>\sum_{i=0}^{n-1}(r_i+r_{i+1})\\
&=r_0+r_n+2\sum_{i=1}^{n-1}r_i
>r_0+r_n.
\end{align*}
This contradicts the closed-star flux identity \eqref{u7:eq:flux}.
\end{proof}

\begin{proof}[Proof of Theorem~\ref{thm:universal}: excluding nonzero row relations]
Fix any legal real-coordinate realization. If the actual target matrix $M_7$ has dependent rows, some $a\ne0$ satisfies $a^{\mathsf T}M_7=0$. Equations~\eqref{u7:eq:row-functional}--\eqref{u7:eq:dual-defect} give a nonzero effective class in $\calD_\Delta$. Choose a representative and impose the two-direction normalization \eqref{u7:eq:two-gauge}, preserving both the actual target functional defined by $a$ and the nonzero class.

Proposition~\ref{u7:prop:scalar-detect} gives nonempty normalized scalar support. The farthest-point circle and tangent in Section~\ref{u7:sec:hull} select a strictly exposed vertex; Proposition~\ref{u7:prop:hull-form} gives its complete convex-hull star. All potentials, scalars, and neighbors come from the original state, so the cyclic construction in Section~\ref{u7:sec:cyclic} applies. Theorem~\ref{u7:thm:local-obstruction} excludes this star using Hermite recurrence, the quadratic companion, and closed-cycle flux, a contradiction.

Thus no $a\ne0$ satisfies $a^{\mathsf T}M_7=0$. All $m$ rows of $M_7$ are independent, and $\Phi_7$ is surjective. Equation~\eqref{u7:eq:epsilon-coker} gives $\eps_7=0$. Using the actual target dimension \eqref{u7:eq:target-dimension} and \eqref{u7:eq:def-epsilon} yields
\[
\dim S_7^2
=L_7^2
=36+15E-m
=36+15E-30I+3N_2+N_3.
\]
The realization was arbitrary, so Theorem~\ref{thm:universal} follows.
\end{proof}

\section{Proof of generic quintic attainment}\label{app:g5}

\subsection{Proof outline: a full-rank realization and generic attainment}\label{zhg5:proof-sketch}

Fix an abstract disk triangulation $\Delta$ admitting a legal straight-line realization. The cofactor criterion in Section~\ref{zhg5:enhanced-map} identifies quintic attainment with surjectivity of the original cofactor map $\Phi_5$. We construct a legal full-rank realization for each combinatorial triangulation and then deduce generic attainment as follows.
\begin{enumerate}
\item \textbf{Pass to an enhanced matrix.} Keep one quadratic cofactor on each physical edge, record its two endpoint values, and record the homogeneous degree-four and degree-five components of the local expansion at every interior vertex. This gives a fixed-size matrix $\mathcal B_\Delta$. Endpoint values control the cubic component; the higher blocks control degrees four and five. Surjectivity of $\mathcal B_\Delta$ therefore implies that of $\Phi_5$; see Lemma~\ref{zhg5:lift}.
\item \textbf{Supply the constraints introduced by insertion.} Three quadratic cofactors define a local map $\Gamma$. Its six endpoint-value or directional-derivative conditions are independent. Eleven modes in their common kernel give a higher-degree matrix of determinant $120\alpha\beta\ne0$. Thus $\Gamma$ is surjective, allowing all new local targets to be specified; see Lemma~\ref{zhg5:local-surjectivity}.
\item \textbf{Return from the collision limit to a legal full-rank realization.} Put $w=a+t\mathbf d$. Source and target projections decompose the collision limit into the old map and the local problem. Their surjectivity gives surjectivity at the limit. Continuity of a nonzero minor and invertibility of the transformations for $t\ne0$ give surjectivity for sufficiently small legal $t>0$. The parameter $t=0$ is used only for a matrix limit.
\item \textbf{Cover every combinatorial triangulation by deletion and restoration.} For triangular outer boundary, delete an interior vertex of valence three, four, or five and retriangulate its link by a fan. Induct on interior vertices, then restore the deleted vertex near the fan root in a legal direction, preserving surjectivity. Complete a general polygonal boundary to a triangle and restrict to a matrix $[\mathcal B_\Delta\ \ 0]$ to obtain a surjective realization of the original disk.
\item \textbf{Use a product polynomial for generic attainment.} Incident-edge collinearity determinants give $D_\Delta$, with $D_\Delta(p)\ne0$ exactly characterizing ordinary realizations. A full-rank realization supplies a nonzero full-row minor $F_\Delta$. Thus $H_{\Delta,5}=D_\Delta F_\Delta\not\equiv0$. Whenever $D_\Delta(p)F_\Delta(p)\ne0$, Lemma~\ref{zhg5:lift} and the cofactor criterion in Section~\ref{zhg5:enhanced-map} give $\dim S_5^2=L_5^2$, as required by the genericity convention in Section~\ref{sec:phase-definitions}. Section~\ref{zhg5:realization-space} details the construction.
\end{enumerate}

The combinatorial induction constructs the required full-rank realization; the product of the ordinary-position polynomial and a nonzero minor witnesses the generic conclusion. Throughout this appendix, smoothness is $C^2$, degree is five, edge forms are cubed, and each physical edge has one shared quadratic cofactor.

\subsection{Realization space and construction of witness polynomials}\label{zhg5:realization-space}

The space $\mathcal R_\Delta$ consists of straight-line planar embeddings: vertices are distinct, triangles are nondegenerate, two simplices intersect only in a common face or not at all, the support is a closed polygonal disk, and no unsplit edge contains another vertex in its relative interior. Here we call these \emph{legal realizations}, distinguishing them from degenerate collision coordinates. Convexity is not required, and consecutive boundary edges may be collinear. All vertices, including boundary vertices, may move. Legality persists under sufficiently small coordinate perturbations, so each legal realization has a legal coordinate neighborhood.

\paragraph{The ordinary-position polynomial $D_\Delta$.}
For the full-rank construction, introduce the auxiliary position condition
\[
g_v(p)=\deg_\Delta(v),
\]
at every interior vertex $v$. Such a realization $p$ is called \emph{ordinary}: no two incident edges at an interior vertex have the same supporting line. Fix the vertex numbering and let $N_\Delta(v)$ be the neighbor set. For $u,w\in N_\Delta(v)$ with $u<w$, set
\[
\omega_{v;u,w}(p)
=\det(p_u-p_v,p_w-p_v)
=(x_u-x_v)(y_w-y_v)-(y_u-y_v)(x_w-x_v).
\]
This factor is nonzero precisely when $vu$ and $vw$ are not collinear. Define
\begin{equation}\label{zhg5:ordinary-polynomial}
D_\Delta(p)=\prod_{v\in V^\circ}
\prod_{\substack{u,w\in N_\Delta(v)\\u<w}}
\omega_{v;u,w}(p).
\end{equation}
An empty product is $1$. Every factor is a nonzero polynomial in the full vertex coordinates. Thus $D_\Delta\not\equiv0$, and for every legal realization $p$,
\[
D_\Delta(p)\ne0
\quad\Longleftrightarrow\quad
g_v(p)=\deg_\Delta(v)\quad\text{for every interior vertex }v.
\]
Only incident edges at interior vertices are constrained; arbitrary triples among all vertices need not be noncollinear.

\paragraph{The surjectivity polynomial $F_\Delta$.}
The enhanced matrix $\mathcal B_\Delta(p)$ below has fixed size and polynomial entries in all vertex coordinates. Geometric induction constructs, for each $\Delta$, a legal $p_*$ with $\mathcal B_\Delta(p_*)$ of full row rank. Let $s_\Delta$ be its row count. Choose $s_\Delta$ columns to form a square submatrix $\mathcal B_{\Delta,J}$ with nonzero determinant at $p_*$, and define
\begin{equation}\label{zhg5:rank-polynomial}
F_\Delta(p)=\det\mathcal B_{\Delta,J}(p),
\qquad F_\Delta(p_*)\ne0.
\end{equation}
Therefore $F_\Delta\not\equiv0$. At any realization with $F_\Delta(p)\ne0$, the enhanced matrix has full row rank.

\paragraph{Generic attainment on $D_\Delta F_\Delta\ne0$.}
Take
\begin{equation}\label{zhg5:witness-polynomial}
H_{\Delta,5}=D_\Delta F_\Delta
\end{equation}
as the witness for genericity in Section~\ref{sec:phase-definitions}. Products of nonzero polynomials are nonzero, so $H_{\Delta,5}\not\equiv0$. On any legal realization $p$ satisfying
\[
D_\Delta(p)F_\Delta(p)\ne0
\]
the factor $D_\Delta(p)\ne0$ guarantees ordinary position and $F_\Delta(p)\ne0$ guarantees enhanced surjectivity. Lemma~\ref{zhg5:lift} and the cofactor criterion in Section~\ref{zhg5:enhanced-map} give
\[
\dim S_5^2(\Delta,p)=L_5^2(\Delta,p).
\]
This is generic attainment in Billera's polynomial definition, with a witness on the full vertex-coordinate space.

We use the elementary fact that a nonzero real polynomial cannot vanish identically on a nonempty open coordinate neighborhood. To avoid finitely many polynomial zero sets inside a legal neighborhood, multiply their nonzero defining polynomials and choose a point where the product is nonzero.

These polynomial conditions are imposed on the full vertex-coordinate space, including boundary coordinates.

\subsection{The enhanced quadratic-cofactor map}\label{zhg5:enhanced-map}

We use the vertex and edge sets fixed in Section~\ref{sec:splines-lower-bound}; in particular, $\Delta_1$ includes boundary edges. The space $\PP_2$ of polynomials of degree at most two has dimension six. At a vertex $v$, let $\HH_{v,k}$ be the degree-$k$ homogeneous polynomials in the translated variable $x-v$. Expand in $x-v$, and let $\tau_v^{4:5}$ extract the homogeneous degree-four and degree-five components. Thus
\[
 \dim(\HH_{v,4}\oplus\HH_{v,5})=5+6=11.
\]
In monomial coordinates, these eleven coefficients correspond to eleven rows of the enhanced matrix.
Choose a nonzero affine equation $\ell_e=0$ for each edge and coherent interior-edge orientations. The original quadratic-cofactor map is
\begin{equation}\label{zhg5:A5}
 \Phi_5:\bigoplus_{e\in E^\circ}\PP_2
 \longrightarrow\bigoplus_{v\in V^\circ}U_{v,5},
 \qquad
 (\Phi_5q)_v=\sum_{e\ni v}\epsilon_{v,e}\ell_e^3q_e,
 \quad U_{v,5}=\sum_{e\ni v}\ell_e^3\PP_2.
\end{equation}
An edge joining two interior vertices has one source $q_e$, appearing with opposite incidence signs at its endpoints. The classical cofactor count identifies the defect above Schumaker's actual lower bound with $\dim\coker \Phi_5$ \citep[Sections 2--3]{MourrainVillamizar13}.

For induction, introduce a quadratic cofactor on every edge, including auxiliary boundary sources, and record its two scalar endpoint values. Define
\begin{equation}\label{zhg5:B}
 \begin{split}
 \mathcal B_\Delta:\bigoplus_{e\in\Delta_1}\PP_2
 &\longrightarrow
 \bigoplus_{v\in V^\circ}(\HH_{v,4}\oplus\HH_{v,5})
 \oplus\bigoplus_{e=[u,v]\in\Delta_1}(\RR\oplus\RR),\\
 (q_e)_e&\longmapsto
 \left(\left\{\tau_v^{4:5}\sum_{e\ni v}\epsilon_{v,e}\ell_e^3q_e\right\}_v,
       \{q_e(u),q_e(v)\}_e\right).
 \end{split}
\end{equation}
Incidence signs in the higher-degree blocks follow coherent vertex-circuit orientations. Boundary-edge columns contribute only to endpoint records; no boundary-vertex compatibility equation is introduced. Source and target dimensions are $6|\Delta_1|$ and $11|V^\circ|+2|\Delta_1|$.

Choose $\ell_e(x)=\det(v-u,x-u)$ for $e=[u,v]$. In source monomials and translated target monomials, $\mathcal B_\Delta$ is a fixed-size coordinate-polynomial matrix, so all its minors are vertex-coordinate polynomials.

\begin{lemma}[Lifting the original cofactor map]\label{zhg5:lift}
If $\mathcal B_\Delta$ is surjective, then $\Phi_5$ is surjective.
\end{lemma}
\begin{proof}
Given $F_v\in U_{v,5}$, choose a local decomposition independently at every interior vertex:
\[
 F_v=\sum_{e\ni v}\epsilon_{v,e}\ell_e^3q_{v,e},
 \qquad q_{v,e}\in\PP_2.
\]
For each endpoint $v$ of each edge, prescribe $q_e(v)=q_{v,e}(v)$ when $v$ is interior, and any value when $v$ is on the boundary. Every incident $\ell_e$ vanishes at $v$, so the cubic homogeneous component of $\sum_e\epsilon_{v,e}\ell_e^3q_e$ in $x-v$ depends only on these values. Prescribe its degree-four and degree-five components to match those of $F_v$.

Surjectivity in \eqref{zhg5:B} provides \emph{one quadratic $q_e$ per edge} satisfying all prescriptions. The assembled polynomial has no terms below degree three at $v$, and its components of degrees three, four, and five agree with $F_v$. Its total degree is at most five, so it equals $F_v$. Discarding boundary-edge auxiliary sources affects no interior target. Thus \eqref{zhg5:A5} is surjective.
\end{proof}

\subsection{The three-cofactor collision map}\label{zhg5:collision-gate}

We prove a local statement involving only quadratic cofactors over characteristic zero. Put
\[
 H=X-Y,\qquad b=(0,\beta),\qquad c=(\alpha,0),
 \qquad\alpha\beta\ne0.
\]
Write $\HH_k(X,Y)$ for the real homogeneous polynomials of degree $k$ in $X,Y$. Use $Z$ as homogenizing variable and identify $f\in\PP_2$ with $Z^2f(X/Z,Y/Z)$. Evaluations and derivatives are taken in the affine chart $Z=1$; $D_{(1,1)}$ is the directional derivative along $(1,1)$. In the homogeneous coefficient basis below, identify $\HH_4\oplus\HH_5$ with $Z\HH_4\oplus\HH_5$. Define
\begin{equation}\label{zhg5:Gamma}
 \begin{split}
 \Gamma:\PP_2^3&\longrightarrow
 \HH_4(X,Y)\oplus\HH_5(X,Y)\oplus\RR^6,\\
 (p,q,\rho)&\longmapsto
 \bigl(\tau^{4:5}[H^3p+X^3q+Y^3\rho],
 p(0),D_{(1,1)}p(0),q(0),q(b),\rho(0),\rho(c)\bigr).
 \end{split}
\end{equation}
Here $\tau^{4:5}$ extracts affine degrees four and five. Source and target dimensions are $18$ and $11+6=17$.

The six scalar coordinates form three pairs. The short edge gives collision records $p(0),D_{(1,1)}p(0)$; the other two merged edge pairs give endpoint-difference coordinates $q(0),q(b)$ and $\rho(0),\rho(c)$. A difference coordinate means one endpoint record restricted to the kernel of the sum of the corresponding records, after retaining that sum as the old coordinate. Section~\ref{zhg5:star-comparison} spells out this invertible choice and its target projection.

\begin{lemma}[Explicit local surjectivity]\label{zhg5:local-surjectivity}
For all nonzero real $\alpha,\beta$, the map $\Gamma$ is surjective.
\end{lemma}
\begin{proof}
The six endpoint functionals map onto $\RR^6$: a value and a directional derivative of $p$ at zero are independent; the two values of $q$ at $0,b$ are independent; likewise those of $\rho$ at $0,c$. The three quadratic sources are independent. It remains to reach the entire eleven-dimensional higher-degree target while all six endpoint functionals vanish.

The common kernel of the six endpoint functionals has dimension twelve.
We select the following eleven modes in this kernel and show that their
higher-degree images form a basis of the eleven-dimensional target.
The first five modes have $Z$-degree one, and the last six are independent
of $Z$:
\begin{equation}\label{zhg5:modes}
 \begin{array}{c|lll@{\qquad}c|lll}
 j&p&q&\rho&j&p&q&\rho\\\hline
 1&0&XZ&0&6&0&X^2&0\\
 2&0&Y(Y-\beta Z)&0&7&0&XY&0\\
 3&HZ&0&0&8&HX&0&0\\
 4&0&0&YZ&9&HY&0&0\\
 5&0&0&X(X-\alpha Z)&10&0&0&Y^2\\
 &&&&11&0&0&XY
 \end{array}
\end{equation}
The higher-degree image of $(p,q,\rho)$ is $H^3p+X^3q+Y^3\rho$. Order the eleven target monomials as
\[
 X^4Z,\ X^3YZ,\ X^2Y^2Z,\ XY^3Z,\ Y^4Z,\quad
 X^5,\ X^4Y,\ X^3Y^2,\ X^2Y^3,\ XY^4,\ Y^5.
\]
Let $M_\Gamma$ be the $11\times11$ coefficient matrix of these eleven higher-degree images in the displayed monomial order. It is block lower triangular:
\[
M_\Gamma=\begin{pmatrix}T_1&0\\ C&T_0\end{pmatrix},
\]
where $C$ is the lower-left block and the diagonal blocks are
\[
 T_1=\begin{pmatrix}
 1&0&1&0&0\\
 0&-\beta&-4&0&0\\
 0&0&6&0&0\\
 0&0&-4&0&-\alpha\\
 0&0&1&1&0
 \end{pmatrix},\qquad
 T_0=\begin{pmatrix}
 1&0&1&0&0&0\\
 0&1&-4&1&0&0\\
 0&0&6&-4&0&0\\
 0&0&-4&6&0&0\\
 0&0&1&-4&0&1\\
 0&0&0&1&1&0
 \end{pmatrix}.
\]
Elementary elimination gives
\begin{equation}\label{zhg5:determinant}
 \det T_1=-6\alpha\beta,\qquad
 \det T_0=-20,\qquad
 \det M_\Gamma=\det T_1\det T_0=120\alpha\beta\ne0.
\end{equation}
Thus the eleven endpoint-zero modes span the full higher-degree target. Interpolate prescribed endpoint data first, subtract the higher-degree image already produced, and correct the remaining error with these modes. This realizes arbitrary endpoint and higher-degree data simultaneously, proving surjectivity of $\Gamma$ without a sign assumption on $\alpha,\beta$.
\end{proof}

\subsection{Exact comparison under low-valence star insertion}\label{zhg5:star-comparison}

Suppose a triangular, quadrilateral, or pentagonal region $P$ in an old legal triangulation is filled by a fan rooted at a boundary vertex $a$ of $P$. Let $b,c$ be its two polygonal neighbors, and $c_j$ the other vertices, if any. Replace the fan by one new interior vertex $w$ joined to every vertex of $P$. This adds one interior vertex and three edges to the full disk, whether $P$ has three, four, or five sides.

The collision comparison requires $a,b,c$ to be noncollinear. This can always be arranged in the generic induction. If the old enhanced map is surjective at a legal realization, both full row rank and legality persist in an open neighborhood. Avoiding finitely many triple-collinearity conditions there preserves both properties. All vertices may move, including boundary vertices. This auxiliary choice is stronger than ordinary position, which only constrains incident directions at interior vertices.

Put $w=a+t\mathbf d$ for small $t>0$, where $\mathbf d$ is parallel to neither $ab$ nor $ac$. After an invertible affine change with origin $a$, assume
\[
 \mathbf d=(1,1),\qquad b=(0,\beta),\qquad
 c=(\alpha,0),\qquad\alpha\beta\ne0.
\]
For each edge $e=[u,v]$, choose a normalized line equation by writing
\[
\ell_e^{\rm raw}=c_e(t)\widehat\ell_e,
\qquad
\widehat q_e=c_e(t)^3q_e^{\rm raw},
\qquad c_e(t)\ne0\quad(t\ne0).
\]
Then $(\ell_e^{\rm raw})^3q_e^{\rm raw}
=\widehat\ell_e^{\,3}\widehat q_e$.
In the original matrix, substitute $q_e^{\rm raw}=c_e(t)^{-3}\widehat q_e$
in its six source columns. Its endpoint records become
$c_e(t)^{-3}\widehat q_e(u)$ and $c_e(t)^{-3}\widehat q_e(v)$;
multiplying these two target rows by $c_e(t)^3$ gives
$\widehat q_e(u),\widehat q_e(v)$.
Thus the normalized enhanced matrix uses
$\widehat\ell_e^{\,3}\widehat q_e$ in its coefficient blocks and
$\widehat q_e(u),\widehat q_e(v)$ as its endpoint records.
All these source and target transformations are invertible for $t\ne0$.

For the short edge, absorb its fixed nonzero orientation factor so that
$\ell_{aw}^{\rm raw}=tH$ and put
\[
p=t^3q_{aw}^{\rm raw}.
\]
Its coefficient contribution is $H^3p$ and its normalized endpoint
records are $p(a),p(w)$. Replace this pair by
\begin{equation}\label{zhg5:divided-difference}
p(a),\qquad\frac{p(w)-p(a)}{t}.
\end{equation}
For $t\ne0$ this is an invertible target transformation, and its entries
converge to $p(a),D_{\mathbf d}p(a)$ as $t\to0$.

For the two merged pairs, choose
\[
\begin{aligned}
\widehat\ell_{ab}&=X,&
\widehat\ell_{wb}&=(1-t/\beta)X+(t/\beta)Y-t,\\
\widehat\ell_{ac}&=Y,&
\widehat\ell_{wc}&=(t/\alpha)X+(1-t/\alpha)Y-t.
\end{aligned}
\]
Their limits are respectively $X,X,Y,Y$.
The other nonshort-edge determinant equations converge to the
corresponding old-edge equations, and their normalized endpoint records
also have finite limits. All cofactors and endpoint records in the
following comparison are written in these normalized coordinates.
Denote the coefficientwise matrix limit by
$\widehat{\mathcal B}_{\rm new}(0)$.

Let $(\mathcal X_{\rm old},\mathcal Y_{\rm old})$ and $(\mathcal X_{\rm new},\mathcal Y_{\rm new})$ be the source and target spaces of the old and new enhanced maps, respectively. Define a surjection $\sigma:\mathcal X_{\rm new}\to\mathcal X_{\rm old}$ by retaining unchanged-edge cofactors and, with coherent orientations, setting
\begin{equation}\label{zhg5:sigma}
 q_{ab}^{\rm old}=q_{ab}^{\rm new}+q_{wb},\qquad
 q_{ac}^{\rm old}=q_{ac}^{\rm new}+q_{wc},\qquad
 q_{ac_j}^{\rm old}=q_{wc_j}.
\end{equation}
Discard the normalized short-edge cofactor $p$. The kernel is represented by three arbitrary quadratics $p,q,\rho$:
\begin{equation}\label{zhg5:source-kernel}
 q_{wa}=p,\qquad q_{wb}=q,\qquad q_{wc}=\rho,\qquad
 q_{ab}^{\rm new}=-q,\qquad q_{ac}^{\rm new}=-\rho,
\end{equation}
All other edge variables vanish. Thus $\ker\sigma\cong\PP_2^3$ has dimension $18$.

Define the target projection $\pi:\mathcal Y_{\rm new}\to\mathcal Y_{\rm old}$. Retain degree-four and degree-five local expansion coefficients at unaffected interior vertices and endpoint records of unchanged edges. For each merged pair in \eqref{zhg5:sigma}, add the corresponding endpoint rows. Deleted fan diagonals $ac_j$ inherit their endpoint rows from $wc_j$ at collision, with the $w$ endpoint relabeled as $a$. If $a$ is interior to the \emph{whole disk}, add the limiting higher-degree blocks at $a,w$; if $a$ is a boundary vertex, discard the block at $w$. Discard both short-edge endpoint rows.

At the collision limit $a=w=0$, denote the four target endpoint records for $ab/wb$ by $(A_0,A_b,B_0,B_b)$, the first pair belonging to $ab$ and the second to $wb$. Choose
\[
S_0=A_0+B_0,\qquad S_b=A_b+B_b,\qquad
D_0=B_0,\qquad D_b=B_b.
\]
This coordinate change is invertible, with
$A_0=S_0-D_0$ and $A_b=S_b-D_b$.
The projection retains $(S_0,S_b)$; on its kernel $S_0=S_b=0$, use
$(D_0,D_b)$ as the endpoint-difference coordinates.
For $ac/wc$, likewise retain the sums and use the two $wc$ records as
endpoint-difference coordinates on the kernel.

These operations give a surjective target projection. Its kernel contains eleven new-star higher-degree coordinates, four endpoint differences from the two split parallel-edge pairs, and two short-edge collision coordinates, so
\[
 \dim\ker\pi=11+4+2=17.
\]

\begin{lemma}[Physical-edge comparison identity]\label{zhg5:comparison}
The transformed limiting matrices form a commutative diagram with exact rows:
\[
 \begin{array}{ccccccccc}
 0&\longrightarrow&\ker\sigma&\longrightarrow&\mathcal X_{\rm new}
   &\xrightarrow{\ \sigma\ }&\mathcal X_{\rm old}&\longrightarrow&0\\
 &&\downarrow\Gamma&&\downarrow\widehat{\mathcal B}_{\rm new}(0)
   &&\downarrow\mathcal B_{\rm old}&&\\
 0&\longrightarrow&\ker\pi&\longrightarrow&\mathcal Y_{\rm new}
   &\xrightarrow{\ \pi\ }&\mathcal Y_{\rm old}&\longrightarrow&0.
 \end{array}
\]
Up to invertible coordinate transformations and sign choices, the induced left map is \eqref{zhg5:Gamma}.
\end{lemma}
\begin{proof}
Away from $a,b,c,w$, unchanged equations converge to the old ones, and $wc_j\to ac_j$ reproduces the old columns with relabeled endpoint rows. At $b,c$, the pairs $ab/wb$ and $ac/wc$ have identical limiting line powers, and their endpoint records add as in \eqref{zhg5:sigma}. If $a$ is interior, adding the $a,w$ blocks cancels the opposite short-edge contributions. If $a$ is on the boundary, the $w$ block is discarded. Therefore
\begin{equation}\label{zhg5:commutation}
 \pi\widehat{\mathcal B}_{\rm new}(0)=\mathcal B_{\rm old}\sigma.
\end{equation}

For a source in \eqref{zhg5:source-kernel}, the \emph{full polynomial} contributions at the remote vertices $b,c$ cancel, not merely selected local homogeneous components. The remaining independent higher-degree block is the new-star block, expressed in normalized coordinates as
\[
 \tau^{4:5}(H^3p+X^3q+Y^3\rho).
\]
If $a$ is interior, the other block is its negative by the kernel condition. With the endpoint coordinates above, the source-kernel values $q_{ab}^{\rm new}=-q$, $q_{wb}=q$ produce
\[
(A_0,A_b,B_0,B_b)=(-q(0),-q(b),q(0),q(b)),
\]
The old coordinates vanish and the differences are $(D_0,D_b)=(q(0),q(b))$. The pair $ac/wc$ similarly gives $\rho(0),\rho(c)$. Thus the four endpoint-difference coordinates are
\[
 q(0),q(b),\rho(0),\rho(c),
\]
The short-edge collision rows are $p(0),D_{(1,1)}p(0)$, giving precisely the six functionals in \eqref{zhg5:Gamma}. No physical interior edge is assigned two independent quadratic sources.
\end{proof}

\begin{corollary}[Legal splitting near the root preserves full row rank]\label{zhg5:split}
Under the noncollinearity and normalization conditions above, if $\mathcal B_{\rm old}$ is surjective, then $\mathcal B_{\rm new}$ is surjective for all sufficiently small positive legal splitting parameters $t$.
\end{corollary}
\begin{proof}
The exact diagram and Lemma~\ref{zhg5:local-surjectivity} lift every collision-limit target. Given $y\in\mathcal Y_{\rm new}$, choose $x_{\rm old}$ with $\mathcal B_{\rm old}x_{\rm old}=\pi y$, then $x_0$ with $\sigma x_0=x_{\rm old}$. The remainder
\[
 r=y-\widehat{\mathcal B}_{\rm new}(0)x_0
\]
satisfies $\pi r=0$ by \eqref{zhg5:commutation}. Since $r\in\ker\pi$, surjectivity of $\Gamma$ gives $k\in\ker\sigma$ with $\widehat{\mathcal B}_{\rm new}(0)k=r$. Thus $x_0+k$ maps to $y$, proving surjectivity at the limit.

Choose a nonzero full-row minor. Continuity at $t=0$ keeps it nonzero for sufficiently small nonzero $t$. All source and row transformations are invertible for $t\ne0$, so the original new enhanced map is surjective.
\end{proof}

\subsection{Induction over all abstract disk triangulations}\label{zhg5:disk-induction}

\begin{lemma}[An interior vertex of low valence]\label{zhg5:low-valence}
A triangulated disk with three outer boundary vertices and at least one interior vertex has an interior vertex of valence three, four, or five.
\end{lemma}
\begin{proof}
The discrete curvature identity for a disk is
\[
 \sum_{v\in V^\circ}(6-\deg_\Delta(v))
 +\sum_{v\in V\setminus V^\circ}(4-\deg_\Delta(v))=6.
\]
With an interior vertex present, each boundary vertex has valence at least three. Otherwise a valence-two boundary vertex forces the outer triangle itself to be a face, leaving no room for interior vertices. The boundary sum is therefore at most three, so the interior sum is positive and some interior valence is at most five. Interior valences in a legal disk are at least three.
\end{proof}

\begin{lemma}[Fan triangulations of short links]\label{zhg5:short-link}
The link of an interior vertex of valence three, four, or five is a simple polygon admitting a legal fan triangulation rooted at one of its vertices.
\end{lemma}
\begin{proof}
The incident triangles fill an embedded disk, so the link is a simple polygon. A triangle is already a fan. A simple quadrilateral has an interior diagonal, giving a fan. A triangulation of a simple pentagon has two noncrossing diagonals; they must share a vertex, since otherwise their four endpoints alternate and the diagonals cross. Every simple polygon has a noncrossing straight-line triangulation, so a legal fan exists even for a nonconvex pentagonal link.
\end{proof}

\begin{proposition}[Generic full row rank with triangular boundary]\label{zhg5:triangle-boundary}
For every finite abstract disk triangulation $\Delta$ with three outer boundary vertices, there is a nonzero coordinate polynomial $F_\Delta$ such that $\mathcal B_\Delta(p)$ is surjective on every legal realization with $F_\Delta(p)\ne0$.
\end{proposition}
\begin{proof}
Induct on $I$. If $I=0$, the disk is a single triangle and the enhanced map consists of the two independent endpoint evaluations of a quadratic on each edge, hence is surjective.

For $I>0$, choose an ordinary legal realization. By Lemma~\ref{zhg5:low-valence}, remove the open star of an interior vertex of valence three, four, or five, and retriangulate its link by a legal fan using Lemma~\ref{zhg5:short-link}. The result is an abstract disk with fewer interior vertices, the same triangular boundary, and a legal realization.

Induction gives a nonzero witness $F_{\rm old}$. Multiply it by the noncollinearity determinants of all old vertex triples; the product is not identically zero. The coordinate choice in Section~\ref{zhg5:realization-space} provides a nonzero point of this product in a sufficiently small neighborhood of the legal fan, preserving strict nonintersection and positive areas. The old enhanced map is surjective and no three old vertices are collinear; in particular $a,b,c$ are noncollinear.

To restore the deleted vertex, let $a$ be the fan root and $b,c$ its polygonal neighbors, and orient the polygon counterclockwise as $a,b,\ldots,c$. Positive fan areas put $a$ strictly inside each supporting half-plane of a side not incident to $a$, and these strict inequalities persist near $a$. For the two incident sides, choose $\mathbf d$ satisfying
\begin{equation}\label{zhg5:legal-cone}
 \det(b-a,\mathbf d)>0,\qquad
 \det(\mathbf d,c-a)>0.
\end{equation}
The adjacent lines are noncollinear, so the intersection of these open directional half-planes is a nonempty cone, even at a reflex root. Avoid finitely many directions causing unwanted collinearities. For sufficiently small $t>0$, the point $w=a+t\mathbf d$ lies in the visibility kernel of the polygon; joining it to consecutive boundary vertices gives a legal nonoverlapping star.

Corollary~\ref{zhg5:split} preserves surjectivity. The restored abstract incidence structure is exactly $\Delta$, though its coordinates need not equal the initially chosen coordinates.

This constructs a legal surjective realization of $\Delta$. Choose a full-row minor nonzero there as $F_\Delta$, as in \eqref{zhg5:rank-polynomial}. Then $F_\Delta\not\equiv0$ and $F_\Delta(p)\ne0$ ensures surjectivity, completing the induction.
\end{proof}

\begin{lemma}[Exterior triangular completion]\label{zhg5:cap}
Every legally embedded disk $\Delta$ is a subcomplex of a finite straight-line triangulation $\widetilde\Delta$ of a larger triangular disk. The completion retains all old vertices and edges without subdivision, has exactly three outer boundary vertices, and adds no edge incident to an old interior vertex. Retaining only the degree-four and degree-five expansion-coefficient rows at old interior vertices and endpoint rows of old edges gives $[\mathcal B_\Delta\ \ 0]$.
\end{lemma}
\begin{proof}
Choose a horizontal coordinate $x$ taking distinct values at the old boundary vertices, and let $u_-$ and $u_+$ be its unique minimum and maximum vertices. Take a triangle $T=ABC$ strictly containing $\Omega$, with unique leftmost and rightmost vertices $A,B$ satisfying $x(A)<x(u_-)$ and $x(B)>x(u_+)$. Except at their old endpoints, the bridges $Au_-$ and $Bu_+$ lie respectively in $x<x(u_-)$ and $x>x(u_+)$. They therefore meet neither $\Omega$ nor each other. The bridges cut the polygonal annulus between $\partial\Omega$ and $\partial T$ into two simple polygons. Triangulate each by diagonals while retaining all marked boundary vertices. Collinear marked vertices may be temporarily omitted and then reinserted by joining them to the opposite vertex of the adjacent triangle. This preserves every old edge without further subdivision and introduces no vertex on the three outer sides. Gluing to $\Delta$ gives $\widetilde\Delta$ with exactly three outer boundary vertices; all new edges lie outside $\Omega$ except at old boundary endpoints, so none is incident to an old interior vertex.

Retain from $\mathcal B_{\widetilde\Delta}$ only the degree-four and degree-five coefficient rows at old interior vertices and endpoint records of old edges. The old columns give exactly $\mathcal B_\Delta$. New edges have no old interior endpoints or retained old-edge endpoint records, so their restricted columns vanish. The restricted matrix is $[\mathcal B_\Delta\ \ 0]$.
\end{proof}

\begin{theorem}[Generic attainment for quintic $C^2$ splines]\label{zhg5:final}
For every finite abstract disk triangulation $\Delta$ admitting a legal straight-line realization, there is a nonzero coordinate polynomial $F_\Delta$ such that every $p\in\mathcal R_\Delta$ with $D_\Delta(p)F_\Delta(p)\ne0$ satisfies
\[
 \dim S_5^2(\Delta,p)=L_5^2(\Delta,p).
\]
The same polynomial condition also gives $\dim S_6^2=L_6^2$.
\end{theorem}
\begin{proof}
For triangular boundary, Proposition~\ref{zhg5:triangle-boundary} gives $F_\Delta$. For any other polygonal boundary, complete the disk using Lemma~\ref{zhg5:cap}. Proposition~\ref{zhg5:triangle-boundary} supplies nonzero $F_{\widetilde\Delta}$. Choose a legal realization in a legal coordinate neighborhood with $F_{\widetilde\Delta}\ne0$, making the larger map surjective. Restriction to old rows and the form $[\mathcal B_\Delta\ \ 0]$ give surjectivity on the induced old coordinates.

We have a legal surjective realization of $\Delta$; choose a full-row minor nonzero there as $F_\Delta$. Construct $D_\Delta$ by \eqref{zhg5:ordinary-polynomial} and put $H_{\Delta,5}=D_\Delta F_\Delta\not\equiv0$. For every legal realization with $D_\Delta(p)F_\Delta(p)\ne0$, Lemma~\ref{zhg5:lift} and the cofactor criterion in Section~\ref{zhg5:enhanced-map} give $\dim S_5^2=L_5^2$. This satisfies Billera's genericity convention in Section~\ref{sec:phase-definitions}, proving Theorem~\ref{thm:generic}. For degree six, define $\Phi_6$ and $U_{v,6}$ by replacing $\PP_2$ with $\PP_3$ in \eqref{zhg5:A5}; since $\PP_3=\PP_2+x\PP_2+y\PP_2$, lifting the three quintic target parts through $\Phi_5$ proves that $\Phi_6$ is surjective under the same condition, hence $\dim S_6^2=L_6^2$ by the cofactor count.
\end{proof}

\clearpage

\begingroup\small\raggedright
\bibliographystyle{plainnat}
\bibliography{references}
\endgroup
\end{document}